\pdfoutput=1
\documentclass[a4paper,onecolumn,oneside]{article}

\usepackage[dvipsnames]{xcolor} % Editorial colors

\usepackage{amsmath,amssymb,amsthm}
\usepackage{mathtools}
\usepackage{thmtools}
\usepackage{enumitem}
\usepackage{dsfont}
\usepackage{titlesec}
\usepackage{tikz-cd}
\usepackage{bm} % Bold accents

\usepackage{url}
\usepackage{hyperref}

\makeatletter
\def\boldsymbol{\mathbf}%
\def\boldsubgap{\hspace{-1.5pt}}%
\def\thmcounter{subsection}

\def\@maketitle{%
  \newpage
  \null
  \vskip 2em%
  \begin{center}%
  \let \footnote \thanks
    {\bfseries\LARGE \@title \par}%
    \vskip 1.5em%
    {\large
      \lineskip .5em%
      \begin{tabular}[t]{c}%
        \@author
      \end{tabular}\par}%
    \vskip 1em%
    {\large \@date}%
  \end{center}%
  \par
  \vskip 1.5em}
\makeatother

\newcommand{\stacks}[1]{\href{https://stacks.math.columbia.edu/tag/#1}{[St:~$\mathrm{#1}$]}}
\newcommand{\arxiv}[2]{\href{https://arxiv.org/abs/#1.#2}{\texttt{arXiv:}\hspace{0pt}\texttt{#1.}\hspace{0pt}\texttt{#2}}}

\renewcommand{\ge}{\geqslant}
\renewcommand{\le}{\leqslant}
\renewcommand{\setminus}{\smallsetminus}

\DeclareMathOperator{\GL}{GL}
\DeclareMathOperator{\Hom}{Hom}

\DeclareMathOperator{\Spec}{Spec}

\DeclareMathOperator{\Gal}{Gal}

\newcommand{\isochar}{\ensuremath{\sim}}
\newcommand{\cdisochar}{\scalebox{2}[1]{\isochar}}
\newcommand{\genisosign}[1]{\smash{\raisebox{-0.65ex}{#1}}}
\newcommand{\isosign}{\genisosign{\isochar}} % Isomorphism sign for x*arrow
\newcommand{\isoarrow}{\xrightarrow{\isosign}}

\newcommand{\tholine}[1]{
  \vbox{%
    \hrule height .08em% thickness
    \kern.25ex%
    \hbox{%
      \kern-.15ex%
      \ensuremath{#1}%
      \kern-.15ex%
    }
  }
}

\newcommand{\bF}{\boldsymbol{F}}
\newcommand{\bG}{\boldsymbol{G}}
\newcommand{\bP}{\boldsymbol{P}}

\newcommand{\cE}{\mathcal{E}}
\newcommand{\cF}{\mathcal{F}}

\newcommand{\cL}{\mathcal{L}}

\newcommand{\cO}{\mathcal{O}}

\newcommand{\Z}{\boldsymbol{Z}}
\newcommand{\Q}{\boldsymbol{Q}}
\newcommand{\R}{\boldsymbol{R}}

\renewcommand{\P}{\bP}

\newcommand{\fn}{\mathfrak{n}}

\newcommand{\fp}{\mathfrak{p}}

\newcommand{\fa}{\mathfrak{a}}

\newcommand{\Fpn}[1]{\bF_{\boldsubgap #1}}
\newcommand{\Fp}{\Fpn{p}}
\newcommand{\Fq}{\Fpn{q}}

\newcommand{\Ho}{H}

\newcommand{\Ga}{\bG_a}

\newcommand{\ab}{{\textup{ab}}}

\newcommand{\unr}{\textit{ur}}

\newcommand{\rambox}[1]{{#1}}
\newcommand{\ram}[1]{\rambox{#1}}

\setlist{itemsep=.125\baselineskip,topsep=.66\baselineskip}

\newlist{statements}{enumerate}{1}
\setlist[statements]{label=\textup{(\arabic*)}}

\newlist{properties}{enumerate}{1}
\setlist[properties]{itemindent=1em}

\titleformat{\section}{\filcenter\normalfont\bfseries\large}{\normalfont\thesection.\ }{0ex}{}
\titleformat{\subsection}{\normalfont\itshape}{\normalfont\thesubsection\ }{0ex}{}

\usetikzlibrary{arrows}
\tikzcdset{arrow style=tikz, diagrams={>={stealth'}}}

\declaretheorem[numberwithin=\thmcounter,name=Theorem]{thm}
\declaretheorem[sibling=thm,name=Lemma]{lem}
\declaretheorem[sibling=thm,name=Corollary]{cor}

\declaretheorem[sibling=thm,style=definition,name=Definition]{dfn}
\declaretheorem[sibling=thm,style=definition,name=Example]{ex}
\declaretheorem[sibling=thm,style=definition,name=Remark]{rmk}

\declaretheorem[name=Theorem]{ithm}

\declaretheorem[numbered=no,style=definition,name=Definition]{dfno}

\DeclareMathOperator{\rank}{rank}

\newcommand{\card}[1]{|#1|}

\DeclareMathOperator{\Lie}{Lie}

\newcommand{\OK}{\cO_K}
\newcommand{\mK}{\fp_K}

\newcommand{\thickbar}[1]{
  \mathchoice{\tholine{#1}}{\tholine{#1}}%
  {\!\scalebox{.65}{\tholine{#1}}\!}%
  {\!\scalebox{.65}{\tholine{#1}}\!}%
}

\newcommand{\fpres}{\thickbar{\fp}}  % Res.char.

\newcommand{\infn}[1]{\|#1\|}
\newcommand{\infnx}[2]{\|#2\|_{#1}}

\newcommand{\AS}{\partial} % Boundary homomorphism symbol

\newcommand{\xtor}[2]{#1[#2\rangle}
\newcommand{\ntor}[1]{\xtor{#1}{\fn}}
\newcommand{\ator}[1]{\xtor{#1}{\fa}}

\newcommand{\Lat}{\varLambda}
\newcommand{\lv}{\lambda}

\newcommand{\ZLat}{L}
\newcommand{\zlv}{v}
\newcommand{\zq}{q}

\newcommand{\F}{F}

\newcommand{\OF}{R}
\newcommand{\mF}{\fp}

\newcommand{\ra}{r}
\newcommand{\Ba}[2]{B(#1,\,#2)}

\newcommand{\ofu}{\pi}
\newcommand{\twu}{\varepsilon}
\newcommand{\Veb}[1]{\Ba{V}{#1}}

\newcommand{\Fcpl}{\smash{\F_{\hspace{-.5pt}\mF}}}

\newcommand{\Vcpl}{\smash{V_{\hspace{-.5pt}\mF}}}

\newcommand{\mdist}{\rho}
\newcommand{\Vmc}{V_{\hspace{-.5pt}\mdist}}

\newcommand{\Latb}[1]{\{\lv\in\Lat,\:\infn{\lv} \leqslant #1\}}
\newcommand{\Lab}[1]{\Ba{\Lat}{#1}}

\newcommand{\fco}{\kappa}
\newcommand{\kinf}{\fco_\infty}
\newcommand{\kinfd}{[\kinf:\fco]}
\newcommand{\Finf}{\F_\infty}
\newcommand{\OFinf}{\cO_{\F,\infty}}
\newcommand{\twc}{c}
\newcommand{\cinf}{c}
\newcommand{\cco}{\card{\fco}}

\newcommand{\Lrat}{V}
\newcommand{\lrv}{v}
\newcommand{\Lrb}[1]{\Ba{\Lrat}{#1}}

\newcommand{\Lcpl}{V_\infty}
\newcommand{\Lcb}[1]{\Ba{\Lcpl}{#1}}
\newcommand{\Lcplq}{W}

\newcommand{\Xstr}{\cO_X}
\newcommand{\Xinf}{\cO_{X,\infty}}
\newcommand{\Etwn}[2]{#1(#2\infty)}%
\newcommand{\Etwneg}[1]{\Etwn{#1}{{-}}}%
\newcommand{\Atwn}[1]{\Etwn{\Xstr}{#1}}%
\newcommand{\Atw}{\Etwn{\Xstr}{}}%
\newcommand{\AtwI}{\Atwn{{-}}}%

\newcommand{\Ugr}{\cE_{\Lat}}%
\newcommand{\Vgrx}[2]{\cE_{#1, #2}}%
\newcommand{\Vgrb}[1]{\Vgrx{\Lat}{#1}}%
\newcommand{\Vgr}{\Vgrb{\ra}}%

\newcommand{\volex}[1]{\operatorname{vol}(#1)}
\newcommand{\volx}[1]{\volex{\Lat,\,#1}}
\newcommand{\volb}{\volx{\ra}}
\newcommand{\volu}{\volex{\Lat}}

\newcommand{\vah}{h}
\newcommand{\rdeg}{f}
\newcommand{\Bu}{\cE}
\newcommand{\econs}{C}
\newcommand{\econx}{3g + 2\rdeg - 1}
\newcommand{\canb}{\Omega}%

\newcommand{\torix}[1]{\textup{ev}_{#1}}%
\newcommand{\tori}{\torix{a}}%

\newcommand{\ad}{\textup{ad}}
\newcommand{\adp}[1]{{#1}_{\ad,\fpres}}
\newcommand{\Aad}{\adp{A}}
\newcommand{\Tad}{\adp{T}}

\newcommand{\rhad}{\rho}
\newcommand{\rhp}{\rho_{\fp}}

\newcommand{\kkum}[2]{[#1,\:#2)}
\newcommand{\kum}[3]{\kkum{#1}{#2}_{#3}}
\newcommand{\kumn}[1]{[\ ,\ )_{#1}}
\newcommand{\kumad}[2]{\kum{#1}{#2}{\ad}}
\newcommand{\kumadv}{[\ ,\ )_{\ad}}
\newcommand{\kumu}[2]{[#1,\ )_{#2}}
\newcommand{\kuml}[1]{[\lv,\ )_{#1}}
\newcommand{\kumadu}[1]{[#1,\ )_{\ad}}
\newcommand{\kumadl}{[\lv,\ )_{\ad}}

\newcommand{\monim}{\varGamma}

\newcommand{\coi}{m}
\newcommand{\lra}{n}

\newcommand{\econd}[1]{\operatorname{\mathfrak{f}}(#1)}
\newcommand{\cond}[1]{\operatorname{\mathfrak{f}}(#1/K)}

\newcommand{\Eg}{{D}}

\begin{document}
\title{Local monodromy of Drinfeld modules}
\author{M.~Mornev}
\date{\small%
\begin{center}%
EPFL SB MATH TN,
Station 8, 1015 Lausanne, Switzerland\\%
\texttt{maxim.mornev@epfl.ch}%
\end{center}%
}
\maketitle
\begin{abstract}
Compared with algebraic varieties the local monodromy of Drinfeld modules
appears to be hopelessly complex: 
The image of the wild inertia subgroup
under Tate module representations is infinite
save for the case of potential good reduction.

Nonetheless we show that Tate modules of Drinfeld modules
are ramified in a limited way: The image of a sufficiently deep
ramification subgroup is trivial. This leads to a new invariant, the
local conductor of a Drinfeld module.
We establish an upper bound on the conductor in
terms of the volume of the local period lattice.

As an intermediate step
we develop a theory of normed lattices in function field arithmetic
including the notion of volume.
We relate normed lattices to vector bundles on projective curves. 
An estimate on Castelnuovo--Mumford regularity of such bundles gives
a volume bound on norms of lattice generators, %in terms of the volume,
and the conductor inequality follows.

Last but not least we describe the image of inertia for Drinfeld modules
with local period lattices of rank $1$. Just as in the theory of local $\ell$-adic Galois representations
this image is commensurable with a commutative unipotent algebraic subgroup.
However in the case of Drinfeld modules such a subgroup can be a product
of several copies of~$\Ga$.
\end{abstract}

\section*{Introduction}

Let $K$ be a local field of positive characteristic
and let $E$ be a Drinfeld module of finite residual characteristic over $\Spec K$
(cf. \S\ref{predm}, \S\ref{reschar} for the terminology).
Pick a separable closure $K^s/K$. %and let $G_K$ be the resulting absolute Galois group.
We will investigate the action of the inertia subgroup $I_K \subset \Gal(K^s/K)$
on the Tate modules $T_\fp E$. % at finite primes $\fp\subset A$.
This parallels the study of local monodromy in
the $\ell$-adic cohomology theory. % of algebraic varieties.

In the $\ell$-adic theory one assumes that the prime $\ell$ is different
from the residual characteristic of the base field.
%The case $\ell = p$ is the domain of Fontaine's theory.
By way of analogy we suppose that %in the setting of Drinfeld modules:
the prime $\fp$ is different from the residual characteristic~$\fpres$ of
the Drinfeld module~$E$.
The case $\fp = \fpres$ belongs to a function field
analogue of Fontaine's theory, and is the subject of a separate study.

Even under the assumption $\fp \ne \fpres$ the local
monodromy of Drinfeld modules is much more complicated than that of algebraic
varieties. The good reduction criterion of Takahashi implies that
the image of the wild inertia in $\GL(T_\fp E)$ is infinite
unless~$E$ has potential good reduction.

However there is a limit on the ramification. % of Tate modules. % $T_\fp E$.
Recall that the inertia subgroup $I_K$ carries a descending filtration by closed normal subgroups $I_K^{\ram{u}}$,
$u \in \Q_{\geqslant 0}$,
the ramification filtration in upper numbering.

\begin{ithm}\label{introbound}
Let $E$ be a Drinfeld module of finite residual characteristic~$\fpres$
over $\Spec K$. Then there is a rational number $u\geqslant 0$ such that
for every prime $\fp\ne\fpres$
the ramification subgroup $I_K^{\ram{u}}$ acts trivially on the Tate module $T_\fp E$.
\end{ithm}

I first deduced this result from a $z$-adic analogue of Grothendieck's
$\ell$-adic monodromy theorem. The $z$-adic monodromy theorem applies
in the generality of arbitrary $A$-motives, and will be published in a forthcoming
article.
A direct and less technically involved proof of Theorem~\ref{introbound} was found by Richard Pink
after some discussions with me.
In this paper I would like to present a revision of Pink's argument (Theorem~\ref{kumimage}~\ref{rambound})
together with related results.

\medskip
%Let $\cO_K\subset K$ be the ring of integers.
Replacing $K$ by a finite separable extension we are free to assume that
the Drinfeld module $E$ has \emph{stable reduction}.
This means that $E$ is an analytic quotient of a Drinfeld module $D$ of good reduction %over $K$ %over $\Spec \cO_K$
by a \emph{period lattice} $\Lat\subset D(K^s)$, cf. \S\ref{conductor}.
Let us further assume that the lattice $\Lat$ is \emph{defined over $K$}, which is to say, $\Lat\subset D(K)$.
Then the action of the inertia group $I_K$ on $T_\fp E$ factors through the maximal abelian quotient $I_K^\ab$.
By Hasse--Arf theorem the induced ramification filtration of $I_K^\ab$
is integrally indexed, so we can make the following definition:

\begin{dfno}
The local \emph{conductor} $\cond{E}$ is the least integer $\coi\ge0$ such that
the ramification subgroup $\smash{I_K^{\ab,\coi+1}}$ acts trivially on $T_\fp E$ for all
$\fp\ne\fpres$.
\end{dfno}

The conductor is zero if and only if $\Lat = 0$ i.e. if the Drinfeld module $E$ has good reduction. % (Lemma \ref{condzero}).
Otherwise the conductor is a positive integer that is prime to~$p$. % (Lemma \ref{condval}).
An important property of the conductor is its invariance under isogenies. % (Lemma \ref{isogdepth}).

The conductor is hard to compute, and we can do it only for special period lattices (Theorem~\ref{lsm}).
In general we can bound the conductor from above by a more amenable invariant.

In the following let us denote by $A$ the chosen \emph{coefficient ring} of our Drinfeld modules, cf. \S\ref{predm}.
This ring is a Dedekind domain over $\Fq$ with finite group of units.
The period lattice $\Lat$ carries a natural structure of a finitely generated projective $A$-module.

Poonen \cite{poonen} introduced the notion of canonical local height.
This provides the period lattice~$\Lat$ with a natural norm.
Building on the work of Taguchi~\cite[\S4]{taguchi} we develop a
theory of normed $A$-lattices in function field arithmetic (\S\ref{lattices}),
and in particular define the volume of such lattices.

\begin{ithm}\label{introcondbound}
Let $E$ be a Drinfeld $A$-module of stable reduction over $\Spec K$
with the local period lattice~$\Lat$ defined over~$K$.
Let $r$ be the rank of~$E$.
Then we have an inequality
\begin{equation*}
\cond{E} \leqslant \volu^s \cdot \econs^{s(r-s)}
\end{equation*}
where $\volu$ is the volume of the period lattice,
$s = r - \rank_A(\Lat)$ 
and $\econs$ is an effective constant that depends only on~$A$.
\end{ithm}

To prove this bound we set up a correspondence between normed $A$-lattices and systems of vector bundles on
the compactification of the curve $\Spec A$.
For such bundles one has a classical criterion of global generation via Castelnuovo--Mumford regularity.
Translated to the setting of lattices this gives a volume bound on norms of generators,
and the conductor inequality follows.

The bound of Theorem~\ref{introcondbound} is effective in the following sense:
There is an algorithm which calculates $\volu$ from the $\tau$-polynomials defining
the Drinfeld module~$E$. This will be demonstrated in a forthcoming article.

Asayama and Huang \cite{asayamahuang,huang} introduced a notion
of conductor for Drinfeld modules of rank~$2$ and proved a version of Szpiro's
conjecture under additional assumptions. It would be interesting
to know the precise relation between the conductor of Asayama--Huang
and the conductor of this paper.
% Observe that our Theorem~\ref{introcondbound} is \emph{not} related to Szpiro's
% conjecture since it is an \emph{upper} bound on the conductor.

Last but not least we study the image of inertia in $\GL(T_\fp E)$.
The following result can be seen as a local open image theorem for
Drinfeld modules with period lattices of rank $1$:
\begin{ithm}\label{introimg}
Let $E$ be a Drinfeld module of finite residual characteristic $\fpres$ and rank~$r$ over $\Spec K$.
Suppose that the local period lattice of $E$ has rank~$1$.
Then for each prime $\fp\ne\fpres$ the image of inertia in $\GL(T_\fp E)$
is commensurable with a closed algebraic subgroup $U \cong (\Ga)^{\times(r-1)}$.
\end{ithm}
Let $V$ be a local $\ell$-adic Galois representation.
The $\ell$-adic monodromy theorem of Grothendieck implies that the image of inertia
in $\GL(V)$ is commensurable with a closed algebraic subgroup $U \cong (\Ga)^{\times d}$,
but the dimension $d$ can only be~$0$ or~$1$.
Theorem~\ref{introimg} shows that in the case of Drinfeld modules the dimension~$d$
can be as large as $\dim(V) - 1$.

The situation with Drinfeld modules that have period lattices of rank at least $2$ %of smaller residual rank %$s \in [2, r-2]$
is made complicated by the presence of nontrivial endomorphisms.
Still one expects that a suitable local open image theorem holds in all cases.
This is a subject of current research.

\bigbreak
Finally let us discuss the proofs of Theorems 1, 2 and 3.
Replacing $K$ with a finite separable extension we are free to assume that the Drinfeld module $E$ is an analytic quotient of a Drinfeld module $D$ of good reduction by a period lattice $\Lat\subset D(K)$. 
For each prime $\fp\ne\fpres$ we then have a short exact sequence of $G_K$-modules
\begin{equation*}
\begin{tikzcd}
0 \rar & T_\fp D \rar & T_\fp E \rar & A_\fp \otimes_A \Lat \rar & 0
\end{tikzcd}
\end{equation*}
with $A_\fp$ the $\fp$-adic completion of $A$.
% Taking the product over all $\fp\ne\fpres$ we obtain a short exact sequence
% \begin{equation*}
% \begin{tikzcd}
% 0 \rar & \Tad D \rar & \Tad E \rar & \Aad \otimes_A \Lat \rar & 0
% \end{tikzcd}
% \end{equation*}
% of so-called \emph{restricted adelic Tate modules}.
The inertia group $I_K$ acts trivially on the first and the last module. % in this sequence,
Consequently, the action of $I_K$ on $T_\fp E$ differs from the identity by a homomorphism $I_K \to \Hom_A(\Lat,\,T_\fp D)$.
This homomorphism can be equivalently seen as a pairing
\begin{equation*}
\Lat \times I_K \to T_\fp D.
\end{equation*}
In the exact same way as in the theory of abelian varieties this pairing turns out to be the restriction to $\Lat$ of a certain universal pairing
\begin{equation*}
\kumn{\fp}\colon D(K) \times I_K \to T_\fp D,
\end{equation*}
the $\fp$-adic \emph{Kummer pairing} of $D$.

% The construction of this Kummer pairing is also the same as in the theory of abelian varieties.
% For every nonconstant element $a \in A$ that is prime to the characteristic of $D$ we have a short exact sequence
% of $G_K$-modules
% \begin{equation*}
% \begin{tikzcd}
% 0 \rar & D[a] \rar & D(K^s) \rar["a"] & D(K^s) \rar & 0.
% \end{tikzcd}
% \end{equation*}
% This sequence gives rise a boundary homomorphism
% \begin{equation*}
% \partial_a\colon D(K^\unr) \twoheadrightarrow H^1(K^\unr,\,D[a])
% \end{equation*}
% with $K^\unr = (K^s)^{I_K}$ the maximal unramified extension of $K$.
% When $a$ is prime to the residual characteristic $\fpres$ of $D$ we further have $D[a] \subset D(K^\unr)$, so the homomorphism $\partial_a$ can be written in the form of a pairing
% \begin{equation*}
% \kumn{a}\colon D(K^\unr) \times I_K \to D[a]
% \end{equation*}
% Taking the limit over all such elements $a$ we obtain a pairing
% \begin{equation*}
% \kumadv\colon D(K^\unr) \times I_K \to \Tad D
% \end{equation*}
% whose target $\Tad D= \prod_{\fp\ne\fpres} T_\fp D$ is the \emph{restricted adelic Tate module} of $D$.
% The $\fp$-adic Kummer pairing mentioned above is the composition of $\kumadv$ and the projection $\Tad D \twoheadrightarrow T_\fp D$.

The Kummer pairing from the theory of abelian varieties factors through the tame quotient of the inertia group $I_K$.
This quotient group is procyclic of order prime to the residual characteristic $p$ of $K$.
By contrast the Kummer pairing of the Drinfeld module $D$ factors through the maximal quotient $J_K = I_K^\ab/(I_K^\ab)^{\times p}$ that is abelian of exponent $p$. Unlike the tame quotient, the induced ramification filtration on $J_K$ is nontrivial, and in fact has a break at every positive integer that is prime to $p$.

We study how the Kummer pairing of $D$ interacts with the ramification filtration on $J_K$.
%One of the main results, Theorem \ref{kumimage}, shows that
We prove that for each element $\lv \in D(K)$ the homomorphism $\kum{\lv}{\ }{\fp}\colon J_K \to T_\fp D$ vanishes on a sufficiently deep ramification subgroup $J_K^u$
with $u$ bounded from above by the Poonen height of $\lv$,
see Theorem~\ref{kumimage}.
This property of the Kummer pairing implies Theorems 1 and 2.

We also prove that the homomorphism $\kum{\lv}{\ }{\fp}\colon J_K \to T_\fp D$ has open image %for every element $\lv\in D(K)$ which
provided that $\lv$ is not contained in the subset of integral elements $D(\cO_K)$, see Theorem \ref{kummeropen}. This property implies Theorem 3.

\vspace{1em}\noindent\textit{The Stacks Project}
%vspace{-.6em}\subsection*{The Stacks Project}

\medbreak\noindent
We will use The Stacks Project \cite{stacks} as a source for algebraic geometry.
References to The Stacks Project
have the form [St:~$wxyz$] where ``$wxyz$'' is a combination of letters and numbers.
The corresponding item is located at \url{https://stacks.math.columbia.edu/tag/}$wxyz$.

%\bigbreak\noindent\textbf{Acknowledgements. }
%vspace*{-.5em}\subsection*{Acknowledgements}
\vspace{1em}\noindent\textit{Acknowledgements}

\medbreak\noindent
I would like to thank
Nihar Gargava,
Stefano Filipazzi,
Ambrus P\'al,
Mihran Papikian,
Federico Pellarin,
Richard Pink,
Maryna Viazovska
and the anonymous reviewers of this paper.
My research is supported by an Ambizione grant of Swiss National Science Foundation 
(project \texttt{PZ00P2\_202119}). The initial investigation was performed when
I was supported by ETH Z\"urich Postdoctoral Fellowship.

This project arose from conversations with Ambrus P\'al
at a workshop of National Center for Theoretical Sciences
in Taipei. I would like to thank NCTS for hospitality.

\section{Normed lattices}\label{lattices}%

Recall that a \emph{lattice} in the sense of geometry of numbers is
a finitely generated free $\Z$-module $\ZLat$
equipped with a positive-definite quadratic form $\zq\colon \ZLat \to \R_{\geqslant 0}$.
The induced norm $\infn{\cdot} = \sqrt{q(\cdot)}$ 
is homogeneous with respect to the archimedean absolute value:
We have $\infn{a \zlv} = |a| \cdot\infn{\zlv}$
for all integers $a \in \Z$ and vectors $\zlv \in \ZLat$.
%$(a,\,\zlv) \in \Z \times \ZLat$.

We shall transfer the notion of a normed lattice to the setting of function fields.
To this end we draw upon an informal analogy between the pair
$(\Q,\,\infty)$ consisting of the field of rational numbers~$\Q$
and its unique archimedean place, % $\infty$, 
and the pair $(F,\,\infty)$
consisting of a global function field~$F$ and an arbitrarily chosen place~$\infty$.

% % The informal analogy between the local field $\Finf$ and the field of real numbers~$\R$
% % \ldots
% % The analogy between $\infty$-adic norms and quadratic forms stretches quite far:
% % Sufficiently good norms can be diagonalized by a variant of Gram--Schmidt process due to Weil
% % (Theorem~\ref{normdiag}).
The ring~$\Z$ is the set of elements $x\in\Q$
which are regular away from~$\infty$ and the field~$\R$ is the 
completion of~$\Q$ at~$\infty$. In the same way the pair $(\F,\,\infty)$ 
determines a Dedekind domain~$A$ and a locally compact field~$\Finf$.
One can consider the pair of rings $A\subset\Finf$ as a function field
equivalent of the pair $\Z\subset\R$.

Inspired by this observation we define an \emph{$A$-lattice} as
a finitely generated projective $A$-module~$\Lat$ equipped
with a norm $\infn{\cdot}\colon \Lat \to \R_{\geqslant 0}$
which is homogeneous with respect to a fixed $\infty$-adic absolute value on $\Finf$
and induces the discrete topology on~$\Lat$. %satisfies a finiteness condition (see Definition~\ref{lattice}).
% %
% The discreteness condition is the function field equivalent of positive-definiteness,
% see the discussion in \S\ref{ilattices}.
% 
The study of normed lattices in function field arithmetic was initiated by Taguchi~\cite[\S4]{taguchi}.
We develop the theory further by relating normed lattices to vector bundles on the 
smooth projective curve~$X$ which has~$F$ as the field of rational functions.
 
As is the case for $\Z$-lattices the norm of an $A$-lattice $\Lat$
extends uniquely to a norm on the ``real'' vector space $\Lcpl = \Finf \otimes_A \Lat$
but this result is rather subtle.
One defines the volume of the lattice~$\Lat$ as a normalized
volume of its fundamental domain in the locally compact vector space~$\Lcpl$.
The relation between lattices and vector bundles
leads to a bound on norms of generators in terms of the volume
(Theorem~\ref{volbound}).

In \cite[\S4]{taguchi} Taguchi defined the discriminant of a lattice.
At the moment it is not clear what is the relation between the volume
as defined in this paper and the discriminant of Taguchi.

\subsection{Normed vector spaces}\label{vecnorms}

As a preparation for the theory of lattices let us review some properties of normed vector spaces following \cite[Ch.~2]{bgr}.
We fix a field $\F$ and an absolute value
$|\hspace{1pt}\cdot\hspace{1pt}|\colon \F \to \R_{\geqslant 0}$
that arises from a non-trivial discrete valuation.
Let $\OF = \{ x \in \F, \,|x| \leqslant 1 \}$ be the corresponding ring of integers.
We denote its maximal ideal by~$\mF$. % and its residue field by~$\kF$.

Let $V$ be a finite-dimensional $\F$-vector space.
Recall that a $\mF$-adic \emph{norm} on~$V$ is a function
%\begin{dfn}\label{vnorm}
%A $\mF$-adic \emph{norm} on~$V$ is a function
$\infn{\cdot}\colon V \to \R_{\geqslant 0}$ %such that
with the following properties:
\begin{properties}[label=(\ensuremath{\mathrm{V}_{\!\arabic*}})]
\item\label{vnormnorm}
$\infn{v} = 0$ if and only if $v = 0$,

\item\label{vnormultra}
$\infn{v + v'} \leqslant \max\{ \infn{v}, \:\infn{v'} \}$,

\item\label{vnormscale}
$\infn{x v} = |x| \cdot \infn{v}$ for all $x \in F$.
\end{properties}
% \end{dfn}
%In the terminology of \cite{bgr} this norm is \emph{faithful};
%however all norms on vector spaces are faithful by \cite[Prop. 2.1.1/4]{bgr}.
We have $\infn{{-}v} = \infn{v}$ by the homogeneity property \ref{vnormscale},
so the ultrametric inequality~\ref{vnormultra}
becomes an equality when $\infn{v}\ne\infn{v'}$. % cf. \cite[Prop. 1.1.1/3 (c)]{bgr}.

\medskip
Let $\ofu \in \F$ be a uniformizer and set $\twu := |\ofu|$.
Observe that $\twu \in (0, 1)$.

\begin{lem}\label{normvalues}
The image of the norm function
$\infn{\cdot}\colon V \to \R_{\geqslant 0}$ has the form
\begin{equation*}
r_1 \twu^\Z \cup \dotsc \cup r_m \twu^\Z \cup \{0\}
\end{equation*}
where
$r_1, \: r_2, \:\dotsc, \: r_m \in (\twu,\:1]$
is a sequence of real numbers, $m\le\dim V$.
\end{lem}
\begin{proof}
The set of values $|\F^\times| = \twu^\Z$ is a subgroup of $\R^\times$,
so the claim follows by \cite[Prop. 2.1.4/2]{bgr}.\end{proof}

%The homogeneity property \ref{vnormscale} and the ultrametric inequality \ref{vnormultra}
The ultrametric inequality \ref{vnormultra} and the homogeneity property \ref{vnormscale}
%The properties \ref{vnormultra} and \ref{vnormscale} of the norm
imply that the balls
\begin{equation*}
\Veb{\ra} = \{v \in V, \, \infn{v} \leqslant \ra \}, \quad \ra \in \R_{>0},
\end{equation*}
are $\OF$-submodules of~$V$ such that
\begin{properties}[label=(\ensuremath{\mathrm{B}_{\arabic*}})]
\item\label{ballscale}
$x\,\Veb{\ra} = \Veb{|x|\,\ra}$ for every scalar $x \in \F^\times$,
  
\item\label{ballex}
$\F \otimes_{\OF} \Veb{\ra} = V$.
\end{properties}

%smallskip
\begin{lem}\label{vballindest}
For each pair of real numbers $\ra' > \ra > 0$ the quotient $\OF$-module
$\Veb{\ra'}/\Veb{\ra}$
has finite length.
\end{lem}
\begin{proof}
It suffices to prove the claim for $\ra = \ra'\twu^n$ with any integer $n>0$.
The $\OF$-module $\Veb{\ra'}$ is torsion-free by construction and has finite rank by Property \ref{ballex}.
We also have $\Veb{\ra} = \mF^n\Veb{\ra'}$ by Property \ref{ballscale}.
Invoking Lemma 3 of \cite[\S5]{poonen} we conclude that the $\OF/\mF^n$-module $\Veb{\ra'}/\Veb{\ra}$ is finitely generated, and thus has finite length. %The claim follows.
\end{proof}

\smallskip
Each norm endows the vector space~$V$ with a topology by means of the fundamental system %of neighbourhoods %of zero %~$0$
$\{\Veb{\ra}\}_{\ra>0}$.
Lemma \ref{normvalues} implies that this system is the same as the system of open balls $\{ v \in V, \, \infn{v} < \ra \}$.

It is important to note that in general the norm topology on~$V$ can be coarser than the canonical $\mF$-adic topology.
An explicit construction of such norms is given in Example \ref{badnorm} below.
From Lemma \ref{vballindest} we directly have

%As an immediate consequence of this lemma we get
\begin{cor}\label{topfin}
%smallskip
%Consequently
A norm on $V$ induces the canonical $\mF$-adic topology if and only if for each real number $\ra > 0$
the $\OF$-module $\Veb{\ra}$ is finitely generated.\qed
\end{cor}

\smallskip
To increase flexibility we will need the notion of a \emph{seminorm}
which is a function $\infn{\cdot}\colon V \to \R_{\geqslant 0}$
satisfying the conditions~\ref{vnormultra} and~\ref{vnormscale}. %of Definition~\ref{vnorm}.
The following claim is easy to check:
\begin{lem}
For each seminorm
$\infn{\cdot}\colon V \to \R_{\geqslant 0}$ 
the following holds:
\begin{statements}
\item
The kernel $H = \{ v \in V, \:\infn{v} = 0 \}$ is an $\F$-vector subspace of~$V$.

\item
The map $\infn{\cdot}$ factors through the quotient homomorphism
$V \twoheadrightarrow V/H$ and induces a norm on $V/H$.\qed
\end{statements}
\end{lem}

\smallskip
Next, let us give a criterion for the norm on $V$ to induce the canonical $\mF$-adic topology.
Let $\Fcpl$ be the $\mF$-adic completion of the field~$\F$ and set $\Vcpl = \Fcpl\otimes_{\F} V$.
%We will relate norms on $V$ to seminorms on the $\mF$-adic completion $\Vcpl = \Fcpl\otimes_{\F} V$ of the vector space~$V$.
% Our fixed absolute value $|\cdot|\colon \F\to\R_{\geqslant 0}$ extends uniquely
% to~$\Fcpl$, and the extended absolute value arises from a discrete valuation on~$\Fcpl$.
% So the considerations of this section apply to the $\mF$-adic completion
% $\Vcpl = \Fcpl \otimes_{\F} V$ of the vector space~$V$.

\begin{thm}\label{normext}
Each norm on the vector space~$V$ extends uniquely to a seminorm on the $\mF$-adic completion $\Vcpl$,
and such an extension is a norm if and only if the original norm induces the $\mF$-adic topology on~$V$.
\end{thm}
\begin{proof}
It follows from Proposition~1 of \cite[II, \S1, n${}^\circ$1]{evt} that every seminorm is continuous with respect to the $\mF$-adic topology on $\Vcpl$ and the analytic topology on $\R_{\geqslant0}$.
Since the vector space $V$ is $\mF$-adically dense in $\Vcpl$ we conclude that the sought extension is unique.

Let $\Vmc$ be the completion of the normed vector space $V$ as in \cite[\S2.1.3]{bgr}.
This is a normed $\Fcpl$-vector space by construction.
The proof of \cite[Prop. 2.3.3/6]{bgr} shows that the natural morphism $\Vcpl \twoheadrightarrow \Vmc$ is surjective.
Taking the composition of the norm on $\Vmc$ with the surjection $\Vcpl \twoheadrightarrow \Vmc$ we obtain the desired seminorm on the vector space~$\Vcpl$.

The seminorm on $\Vcpl$ is a norm if and only if the surjection $\Vcpl\twoheadrightarrow\Vmc$ is an isomorphism. According to \cite[Prop. 2.3.3/6]{bgr} the latter happens if and only if the norm topology on $V$ is the canonical $\mF$-adic topology, as claimed.
\end{proof}

%An important special case of Theorem \ref{normext} will be referred to separately:
%For later use, let us state a special case of Theorem \ref{normext}:

\begin{cor}\label{complballft}
Suppose that the field $\F$ is $\mF$-adically complete.
Then the norm topology on $V$ is the canonical $\mF$-adic topology.\qed
\end{cor}

\smallskip
The behaviour of $\mF$-adic norms in Theorem \ref{normext} resembles the behaviour of real-valued quadratic forms on a finite-dimensional $\Q$-vector space~$V$.
Every such form $q$ extends uniquely to the real vector space $V_\infty = \R\otimes_{\Q} V$.
If the form~$q$ takes strictly positive values on~$V\setminus\{0\}$ then it is positive-semidefinite, % on~$V_\infty$,
and the map $\zlv \mapsto q(\zlv)^{1/2}$ is an archimedean seminorm on~$V_\infty$.
This seminorm is a norm if and only if the form~$q$ is positive-definite.
Comparing to Theorem~\ref{normext} we conclude that
the property of inducing the $\mF$-adic topology %having finite type %Definition~\ref{normfintype}
%is the equivalent of positive-definiteness for $\mF$-adic norms.
can be understood as a form of positive-definiteness for $\mF$-adic norms.

%medskip
As an offshoot of Theorem~\ref{normext} we get a description of all norms that fail to induce the $\mF$-adic topology on~$V$:

\begin{ex}\label{badnorm}
Suppose that $\dim V \geqslant 2$, and that the field $\F$ is not $\mF$-adically complete.
Pick a nonzero subspace $H \subset \Vcpl$ which is totally irrational in the sense that $H \cap V = \{0\}$.
Pick a norm on the quotient $\Vcpl/H$.
The natural morphism $V\hookrightarrow \Vcpl/H$ is injective
by construction, so we obtain a norm on~$V$ by composition
with the chosen norm on~$\Vcpl/H$.
% Pick a seminorm on $\Vcpl$ with kernel $H\ne0$ which is totally irrational in the sense that $H \cap V = \{0\}$.
% This seminorm restricts to a norm on the subspace $V\subset\Vcpl$.

Theorem~\ref{normext} implies that the norm toplogy on $V$ is not $\mF$-adic for otherwise the subspace~$H$ will be zero.
It follows by Corollary \ref{topfin} that the balls $\Veb{\ra}$ are not finitely generated as $\OF$-modules.
Theorem~\ref{normext} also implies that varying the choices of~$H$ and of the norm on $\Vcpl/H$
%choice of the seminorm on $\Vcpl$
one obtains every norm on~$V$ that does not induce the $\mF$-adic topology.
\end{ex}

\subsection{The setting}\label{coeffring}

From now on we fix a global function field $\F$ and a place $\infty$ of~$\F$.
We will use the following notation:
\begin{itemize}
\item
$\fco \subset \F$ is the ring of elements which are regular at all places,

\item
$A \subset \F$ is the ring of elements which are regular outside~$\infty$,

\item
$\OFinf \subset \F$ is the ring of elements which are regular at~$\infty$,

\item
$\Finf$ is the $\infty$-adic completion of~$\F$,

\item
$\kinf$ is the residue field of $\Finf$ (and of $\OFinf$),

\item
$\twc = |\kinf|$ is the cardinality of this residue field.
\end{itemize}

Fix an $\infty$-adic absolute value
$|\cdot|_\infty\colon \Finf \to \R_{\geqslant 0}$
such that
%\begin{equation*}
$|\ofu|_\infty^{-1} = \twc$
%\end{equation*}
for a uniformizer $\ofu\in\Finf$.
Although our theory works with any normalization,
this particular choice %makes the formulas cleaner.%
leads to better-looking formulas.
% \footnote{%
% These are the formulas of Lemma~\ref{voldet}, Theorem~\ref{latcount} and the main Theorem~\ref{volbound}.%
% }

The ring $\fco$ is a finite field, called the \emph{field of constants} of~$\F$.
The ring $A$ is a Dedekind domain of finite type over~$\fco$. %Note that $A^\times = \fco^\times$. %~$A^\times$.
We will refer to $A$ as the \emph{coefficient ring}
and it will serve us as an analog of the ring of integers~$\Z$.
The local field $\Finf$ will play the role of the field of real numbers~$\R$.

\subsection{Lattices}\label{ilattices}
Let $\Lat$ be a finitely generated projective $A$-module.
%
%\smallskip\begin{dfn}
An $\infty$-adic \emph{norm} on $\Lat$ is a function
$\infn{\cdot}\colon \Lat \to \R_{\geqslant 0}$
such that:
\begin{properties}[label=($\Lambda_{\arabic*}$)]
\item\label{latdzero}
$\infn{\lv} = 0$ if and only if $\lv = 0$,

\item\label{latdultra}
$\infn{\lv + \lv'} \leqslant \max\{ \infn{\lv}, \:\infn{\lv'} \}$,

\item\label{latdscale}
$\infn{a\lv} = |a|_\infty \, \infn{\lv}$ for all $a \in A$.
\end{properties}
%\end{dfn}
%
%medskip
Following the construction in \cite[\S2.1.3]{bgr} every such norm %on the $A$-module $\Lat$ 
extends uniquely to an $\infty$-adic norm on the rational vector space
\begin{equation*}
\Lrat = F \otimes_A \Lat.
\end{equation*}
We will apply the considerations of \S\ref{vecnorms} to %the discrete valuation ring $\OFinf$ and
the normed vector space $\Lrat$.
By analogy with \S\ref{vecnorms} we define the following subsets of~$\Lat$:
\begin{equation*}
\Lab{\ra} = \Latb{\ra}, \quad \ra \in \R_{>0}.
\end{equation*}
This time $\Lab{\ra}$ is merely a $\fco$-vector space.

\begin{lem}\label{ballindest}
For each pair of real numbers $\ra' > \ra > 0$ 
the subspace $\Lab{\ra}$ has finite index in $\Lab{\ra'}$.
\end{lem}
\begin{proof}
By construction
$\Lab{s} = \Lat \cap \Lrb{s}$ for all $s > 0$. We thus have an inclusion of
$\fco$-vector spaces
\begin{equation*}
\frac{\Lab{\ra'}}{\Lab{\ra}} \hookrightarrow \frac{\Lrb{\ra'}}{\Lrb{\ra}},
\end{equation*}
and the claim follows from Lemma~\ref{vballindest}.
\end{proof}

%medskip
We are ready for the main definition of this section:

\begin{dfn}\label{lattice}
An \emph{$A$-lattice} is a finitely generated projective $A$-module $\Lat$
supplied with an $\infty$-adic norm $\infn{\cdot}$ that induces the discrete topology on $\Lat$.
\end{dfn}

%medskip
In other words there is a real number $\varepsilon>0$ such that
every nonzero lattice vector $\lv$ satisfies
$\infn{\lv} \geqslant \varepsilon$.
By Lemma~\ref{ballindest} this holds if and only if the subsets $\Latb{\ra}$
are finite for every $\ra > 0$.

A norm can induce a non-discrete topology on $\Lat$. % when the rank of $\Lat$ is at least~$2$.
All such norms are described in Example~\ref{badnorm} above.
%Such is the case for the norms of Example \ref{badnorm} above.

%medskip
The discreteness condition of Definition~\ref{lattice}
%is
can be seen as a function field analog of positive-definiteness. % of quadratic forms.
Let $q$ be a real-valued quadratic form on
a finitely generated free $\Z$-module~$\ZLat$.
Suppose that the form $q$ takes strictly positive values on~$\ZLat \setminus \{0\}$.
Then~$q$ is positive-semidefinite as a quadratic form on the real vector space $\R\otimes_\Z \ZLat$,
and is positive-definite if and only if for each real number $\ra > 0$
the set $\{\zlv \in \ZLat, \,q(\zlv) \leqslant \ra \}$ is finite,
cf. Lemma~9.5 of \cite[VIII.9]{silverman}.

\begin{thm}\label{normex}
Let $\Lat$ be a lattice.
Then the norm on~$\Lat$ extends to a unique $\infty$-adic norm on the vector space $\Lcpl = \Finf\otimes_A \Lat$.
\end{thm}
\begin{proof}
We have observed above that the norm on~$\Lat$ extends uniquely
to an $\infty$-adic norm on the rational vector space~$\Lrat = \F\otimes_A \Lat$.
By Theorem~\ref{normext}
the norm on~$\Lrat$ extends to a unique seminorm on~$\Lcpl$.
Let $H$ be the kernel of this seminorm 
and consider the quotient vector space $\Lcplq = \Lcpl/H$.
The seminorm on $\Lcpl$ is the composite of the surjection
$\Lcpl \twoheadrightarrow \Lcplq$ and a norm on~$\Lcplq$.

The lattice $\Lat$ injects into the quotient vector space $\Lcplq$.
By assumption the sets $\Lab{\ra}$ are finite for all $\ra > 0$
so $\Lat$ is discrete in $\Lcplq$ with respect to the norm topology.
As the field $F_\infty$ is $\infty$-adically complete it follows by Corollary \ref{complballft}
that the norm topology on $W$ is the canonical $\infty$-adic topology.
The fact that $\Lat$ is discrete in $W$ thus implies that
$\rank \Lat \leqslant \dim \Lcplq$.
As $\rank\Lat = \dim \Lcpl$ we deduce that $H = 0$.
\end{proof}

\begin{cor}\label{latnormft}
Let $\Lat$ be a lattice, and consider the rational vector space $V = F\otimes_A \Lat$ with the induced norm.
Then for each real number $\ra > 0$ the ball $\Lrb{\ra}$ is finitely generated as a module over the ring~$\OFinf$.
\end{cor}
\begin{proof}
The norm on $V$ extends to a norm on $V_\infty$ by Theorem \ref{normex},
so invoking Theorem \ref{normext} we derive that the norm topology on~$V$ is the canonical $\infty$-adic topology.
The claim then follows by Corollary \ref{topfin}.
\end{proof}

\subsection{Lattices and vector bundles}

We shall relate lattices to vector bundles on the smooth projective curve $X$ over $\Spec \fco$
which compactifies the affine curve $Y = \Spec A$.
By construction we have $X\setminus Y = \{\infty\}$.
Let us denote by $\AtwI$ the ideal sheaf of the
reduced closed subscheme $\{\infty\} \subset X$.

Pick a lattice $\Lat$, set
$\Lrat = \F\otimes_A \Lat$,
and let $\Ugr$ be the locally free sheaf on the curve~$Y$ induced by the $A$-module~$\Lat$.
Let $\ra > 0$ be a real number.

\begin{dfn}
The quasi-coherent sheaf~$\Vgr$ on the curve~$X$
is constructed by gluing the locally free sheaf~$\Ugr$
to the $\OFinf$-module $\Lrb{\ra}$
via the canonical isomorphism
\begin{equation*}
\F\otimes_A \Lat \isoarrow V \isoarrow \F\otimes_{\OFinf} \Lrb{\ra}.
\end{equation*}
\end{dfn}

\begin{thm}\label{latbun}%
The sheaf $\Vgr$ has the following properties:
\begin{statements}
\item\label{latbunbun}
The sheaf $\Vgr$ is locally free of finite rank.

\item\label{latbunglob}
$\Ho^0(X,\:\Vgr) = \Latb{\ra}$.

\item\label{latbuntwist}
For every $i \in \Z$ we have $\Etwn{\Vgr}{i} = \Vgrb{\ra\twc^i}$.
\end{statements}
\end{thm}
\begin{proof}
Property~\ref{latbunbun} follows from Corollary~\ref{latnormft}
and Property~\ref{latbunglob} is immediate. %from the definition.
To show Property~\ref{latbuntwist} note that the locally free sheaf
$\Etwn{\Vgr}{i}$ is obtained by gluing the sheaves induced by the modules $\Lat$ and $\ofu^{-i} \Lrb{\ra}$
with $\ofu$ a uniformizer of $\OFinf$.
The claim follows since
$\ofu^{-i} \Lrb{\ra} = \Lrb{\ra\twc^i}$.
\end{proof}

Combining Theorem \ref{latbun} with Lemma \ref{normvalues} we obtain a correspondence between lattices and systems of vector bundles:

\begin{cor}
%The data of an $A$-lattice amounts to the data of a finite sequence of real numbers
There is a one-to-one correspondence between $A$-lattices and
pairs which consist of
%\begin{enumerate}[label=\textup{(}\alph*\kern1pt\textup{)}]
%\item
(i) a strictly increasing chain of locally free sheaves 
\begin{equation*}
\cE_1 \subset \dotsc \subset \cE_m \subset \Etwn{\cE_1}{}
\end{equation*}
on the complete curve~$X$ such that
$\cE_i|_Y = \cE_{i+1}|_Y$ for all $i \in \{1, \dotsc, m -1 \}$,
and
(ii) a sequence of real numbers
\begin{equation*}
\{ r_1 < \dotsc < r_m \} \subset (\twc^{-1}, 1].
\end{equation*}
\end{cor}

The $A$-module $\Lat = \Ho^0(Y,\,\cE_i)$ is finitely generated projective and is independent of the choice of~$i$.
It carries a filtration by $\fco$-vector spaces
$\Ho^0(X,\,\Etwn{\cE_i}{j})$ with $i \in \{1,\dotsc,m\}$ and $j \in \Z$.
For each element $\lv \in \Lat$
we set
\begin{equation*}
\infn{\lv} = \inf \{r_i \twc^j \:|\: \lv \in \Ho^0(X,\,\Etwn{\cE_i}{j})\}.
\end{equation*}
It is easy to check that the function $\infn{\cdot}$ is a norm on $\Lat$.
The normed module $\Lat$ is a lattice since the vector spaces
$\Ho^0(X,\,\Etwn{\cE_i}{j})$ are finite-dimensional.

% One can express the locally free sheaves $\Vgr$ by means of the Proj construction.
% To this end let us recall a description of the compactified curve~$X$ 
% given by Drinfeld \cite[\S1]{drinfeld-commrgs}.
% %
% The coefficient ring~$A$ carries an increasing filtration
% by $\fco$-vector subspaces
% \begin{equation*}
% \Agn{i} = \{ a \in A,\: |a|_\infty \leqslant \twc^i \}, \quad i \in \Z_{\geqslant 0},
% \end{equation*}
% with the properties that $1 \in \Agn{0}$ and $\Agn{i} \cdot \Agn{j} \subset \Agn{i+j}$.
% The associated graded algebra
% $\Agr = \bigoplus_{i\geqslant 0} \Agn{i}$ determines the curve~$X$
% %by means of the Proj construction:
% as follows:
% \begin{equation*}
% X = \operatorname{Proj} \Agr.
% \end{equation*}
% The inclusions $A^{i-1} \subset A^{i}$ define a graded
% morphism $\Agr[-1] \hookrightarrow \Agr$ which
% identifies the ideal sheaf~$\AtwI$ with $\Proj \Agr[-1]$.
% 
% The graded global sections functor $\cF \mapsto \bigoplus_{i\geqslant 0} \Ho^0(X,\:\Etwn{\cF}{i})$
% %of~\stacks{01MM}
% transforms the sheaf~$\Vgr$ to the graded $\Agr$-module 
% \begin{equation*}
% \Egr = \bigoplus\nolimits_{i \geqslant 0} \Lab{\ra \twc^i}.
% \end{equation*}
% This module is finitely generated~\stacks{0B5Q}.
% We have a canonical isomorphism
% \begin{equation*}
% \Proj\Egr \isoarrow \Vgr,
% \end{equation*}
% see~\stacks{0AG5}.
% %This module is torsion-free since the lattice~$\Lat$ is torsion-free as an $A$-module.
% One can prove the finite generatedness of~$\Egr$ without relying on Theorem~\ref{normex}
% and thus define $\Vgr$ via the Proj construction.

\subsection{The volume}

Let $\Lat$ be a lattice, and consider the vector space $\Lcpl = \Finf \otimes_A \Lat$ with the induced norm.
%which exists and is unique by Theorem \ref{normex}).
By Corollary \ref{complballft} the norm topology on $V_\infty$ is the canonical $\infty$-adic topology.
For each real number $\ra > 0$ the ball $\Lcb{\ra}$ is thus compact and open, and we can make the following definition:

\begin{dfn}
The $\ra$-normalized \emph{volume} of $\Lat$ is defined by the formula
\begin{equation*}
\volb = \mu(V_\infty/\Lat)
\end{equation*}
where $\mu$ is the unique translation-invariant Haar measure on~$V_\infty$
satisfying
$\mu(\Lcb{\ra}) = 1$. 
%\end{equation*}
In the case $\ra = 1$ we will use a simplified notation:
\begin{equation*}
\volu = \volx{1}.
\end{equation*}
\end{dfn}

Denoting by~$n$ the rank of the lattice~$\Lat$
we have an equality for every integer~$i$:
\begin{equation*}
\volx{\ra\twc^i} = \twc^{-ni} \, \volb.
\end{equation*}
This results from the fact that
%\begin{equation*}
$\mu(\Lcb{\ra\twc^i}) = \twc^{ni} \,\mu(\Lcb{\ra})$
%\end{equation*}
for every translation-invariant Haar measure $\mu$. % on the vector space~$V_\infty$.

In the following we denote by $\chi$ the Euler characteristic of coherent sheaves
on the curve~$X\to\Spec\fco$.

\begin{lem}\label{voleuler}
$\volb = {\cco}^{-\chi(\Vgr)}$.
\end{lem}
\begin{proof}
The cohomology of the locally free sheaf $\Vgr$ is computed by a \v{C}ech
complex
\begin{equation*}
\begin{tikzcd}[column sep=5em]
\big[ \Lat \oplus \Lcb{\ra} \rar["\textup{difference}"] &\Lcpl \big]
\end{tikzcd}
\end{equation*}
concentrated in degrees~$0$ and~$1$. We rewrite this complex as follows:
\begin{equation*}
\begin{tikzcd}[column sep=5em]
\big[ \Lcb{\ra} \rar &\Lcpl/\Lat \big].
\end{tikzcd}
\end{equation*}
Thus for every translation-invariant Haar measure $\mu$ on $\Lcpl$
we have an equality:
\begin{equation*}
\frac{\mu(\Lcpl/\Lat)}{\mu(\Lcb{\ra})} = \frac{\card{\Ho^1(X,\,\Vgr)}}{\card{\Ho^0(X,\,\Vgr)}}.
\end{equation*}
The claim follows.
\end{proof}

As in the case of $\Z$-lattices we can compute the volume of an $A$-lattice
via the absolute value of a determinant.
In the following formula we treat the respective wedge products
%$\lv_1 \wedge \dotsc \wedge \lv_n$ and $\lrv_1 \wedge \dotsc \wedge \lrv_n$
as elements of one-dimensional $\Finf$-vector space $\det(\Lcpl)$.
Their quotient is a well-defined element of $\Finf^\times$.

\begin{lem}\label{voldet}
Pick a basis $\lrv_1, \dotsc, \lrv_n$ of the free module $\Lcb{\ra}$
and pick vectors $\lv_1,\dotsc,\lv_n$ generating a submodule of finite index $e$ in~$\Lat$.
Then
\begin{equation*}
\volb = \frac{{\cco}^{(g-1)n}}{e} \:
\left|\frac{\lv_1 \wedge \dotsc \wedge \lv_n}{\lrv_1 \wedge \dotsc \wedge \lrv_n} \right|_\infty%^{\renex}
\end{equation*}
where $g$ is the genus of the coefficient field~$\F$.
\end{lem}
\begin{proof}[Proof of Lemma~\ref{voldet}]
The product $s = \lv_1 \wedge \dotsc \wedge \lv_n$ is a rational
section of the invertible sheaf $\det(\Vgr)$. Its order at the point~$\infty$ equals
\begin{equation*}
\kinfd \: v_\infty\!\!\left(\frac{\lv_1 \wedge \dotsc \wedge \lv_n}{\lrv_1 \wedge \dotsc \wedge \lrv_n} \right).
\end{equation*}
where $v_\infty\colon \Finf^\times\twoheadrightarrow\Z$ is the normalized $\infty$-adic valuation.
Hence
\begin{equation*}
{\cco}^{-\deg(s)} = \frac{1}{e} \cdot 
\left|\frac{\lv_1 \wedge \dotsc \wedge \lv_n}{\lrv_1 \wedge \dotsc \wedge \lrv_n} \right|_\infty.
\end{equation*}
The claim then follows from Lemma~\ref{voleuler} and Riemann--Roch formula.
\end{proof}

\begin{thm}\label{latcount}
Let $\Lat$ be a lattice of rank~$n$ and let $\ra > 0$ be a real number.
Then for all integers $i \gg 0$ we have an equality:
\begin{equation*}
\card{\{\lv \in \Lat, \: \infn{\lv} \leqslant \ra\twc^i\}} = \frac{\cinf^{ni}}{\volb}.
\end{equation*}
Moreover for all integers $i$ we have a lower bound:
\begin{equation*}
\card{\{\lv \in \Lat, \: \infn{\lv} \leqslant \ra\twc^i\}} \geqslant \frac{\cinf^{ni}}{\volb}.
\end{equation*}
\end{thm}
\begin{proof}
Let $i$ be a non-negative integer.
Consider a natural short exact sequence of sheaves on the curve~$X$:
\begin{equation*}
\begin{tikzcd}
0 \rar & \Vgr \rar & \Etwn{\Vgr}{i} \rar & \cF \rar & 0.
\end{tikzcd}
\end{equation*}
The coherent sheaf $\cF$ is concentrated at the point $\infty$
and has length $n i$ as an $\Xinf$-module. We thus have
an equality
\begin{equation*}
\frac{\card{\Ho^0(X,\:\Etwn{\Vgr}{i})}}{\card{\Ho^1(X,\:\Etwn{\Vgr}{i})}} =
\cinf^{ni} \, \frac{\card{\Ho^0(X,\:\Vgr)}}{\card{\Ho^1(X,\:\Vgr)}}.
\end{equation*}
By construction we have $\Ho^0(X,\:\Etwn{\Vgr}{i}) = \Latb{\ra\twc^i}$.
Hence Lemma~\ref{voleuler} implies that
\begin{equation*}
\frac{\card{\Latb{\ra\twc^i}}}{\card{\Ho^1(X,\:\Etwn{\Vgr}{i})}} =
\frac{\cinf^{ni}}{\volb}.
\end{equation*}
The cohomology group $\Ho^1(X,\:\Etwn{\Vgr}{i})$ vanishes for $i \gg 0$
as the invertible sheaf $\Atw$ is ample. The first claim of the theorem follows.
To prove the second claim note that $\twc^{ni} \,\volb^{-1} = \volx{\ra\twc^i}^{-1}$
for each integer~$i$. It is thus enough to treat the case $i = 0$ which
follows from the displayed formula above.
\end{proof}

\begin{cor}
Every lattice $\Lat$ of rank~$n$ contains a nonzero vector of norm at most
$\twc\cdot{\volu}^{1/n}$.\qed
\end{cor}

\subsection{A bound on norms of generators}

We would like to find a real number $\ra > 0$ such that
the subset $\Latb{\ra}$ generates the lattice $\Lat$.
To this end we will employ the correspondence between lattices and vector bundles.
We will use supplementary notation:
\begin{itemize}
\item
$g$ is the genus of the smooth projective curve~$X\to\Spec\fco$,

\item
$\rdeg = \kinfd$.
\end{itemize}

The invertible sheaf $\Atw$ has degree $\rdeg$.
%Pick an integer $\vah$ such that the twist $\Atwn{\vah}$ is very ample.
Let us also pick an auxilliary integer $\vah$ such that the twist $\Atwn{\vah}$ is very ample.
We will utilize a classical criterion for global generation of coherent sheaves on curves:
\begin{lem}\label{cmr}
Let $\Bu$ be a coherent sheaf on the curve~$X$.
If $\Ho^1(X,\:\Bu) = 0$ then the sheaf $\Etwn{\Bu}{\vah}$ is globally generated.
\end{lem}
\begin{proof}
Let $\iota\colon X \hookrightarrow \P^N$ be the closed embedding
defined by global sections of $\Atwn{\vah}$. The coherent sheaf
$\iota_*(\Bu)$ is $1$-regular in the sense of Castelnuovo--Mumford \stacks{08A3}.
Hence the sheaf $\iota_*(\Bu)(1) = \iota_*(\Etwn{\Bu}{\vah})$ is globally generated \stacks{08A8}.
\end{proof}

\begin{lem}\label{cmrsub}
Let $\Bu$ be a locally free sheaf of rank~$n$ on the curve~$X$. Set
\begin{equation*}
i = \left\lceil\frac{-\deg\Bu + (2g - 2)n + 1}{\rdeg}\right\rceil + \vah.
\end{equation*}
Suppose that the sheaf $\Bu$ is an iterated extension of invertible sheaves
of degree at most~$2g - 2$. 
Then the sheaf $\Etwn{\Bu}{i}$ is globally generated.
\end{lem}
\begin{proof}
Write the degrees of invertible sheaves in the form $(2g - 2) - e_j$
with $e_j \geqslant 0$.
We have $\deg\Bu = (2g-2)n - \sum e_j$ so that
\begin{equation*}
-e_j - \deg \Bu + (2g-2)n \geqslant 0.
\end{equation*}
Hence the sheaf $\Bu' = \Etwn{\Bu}{(i-\vah)}$
is an iterated extension of invertible sheaves of degree \emph{at least} $2g-1$.
Induction on the number of invertible sheaves implies that
$\Ho^1(X,\,\Bu') = 0$
so the sheaf $\Etwn{\Bu'}{\vah}$ is globally generated by Lemma~\ref{cmr}.
\end{proof}

\begin{lem}\label{cmrvan}
Let $\Bu$ be a locally free sheaf of rank~$n$ on the curve~$X$. Set
\begin{equation*}
i = \left\lceil\frac{-\deg\Bu + (2g-2)n + 1}{\rdeg}\right\rceil + \vah n.
\end{equation*}
Suppose that $\Ho^0(X,\,\Bu) = 0$.
Then the sheaf $\Etwn{\Bu}{i}$ is globally generated.
\end{lem}
\begin{proof}%
Let $\canb$ be the canonical sheaf of the curve $X\to\Spec\fco$
and set $\cF = \Bu^* \otimes \canb$.
Serre duality shows that $\Ho^1(X,\,\cF) = 0$ %by Serre duality 
so the sheaf $\Etwn{\cF}{\vah}$ is globally generated
by Lemma~\ref{cmr}. Dualizing and twisting we get an embedding
$\Etwn{\Bu}{-\vah} \hookrightarrow \canb^{\oplus m}$, $m \gg 0$.
Hence the sheaf
$\Bu' = \Etwn{\Bu}{-\vah}$
is an iterated extension of invertible subsheaves $\cL \subset \canb$.
Every such subsheaf has degree at most $2g - 2$.
Lemma~\ref{cmrsub} implies that the sheaf $\Etwn{\Bu'}{(i+\vah)} = \Etwn{\Bu}{i}$
is globally generated with
\begin{equation*}
i = \left\lceil\frac{-\deg\Bu'+(2g-2)n + 1}{\rdeg}\right\rceil.
\end{equation*}
The claim follows since
$\deg\Bu' = \deg\Bu - \rdeg\vah n$.
\end{proof}

\begin{thm}\label{volbound}
Let $\Lat$ be a lattice of rank~$n$. % and let $\ra > 0$ be a real number.
Suppose that every nonzero vector of $\Lat$ has norm at least~$1$.
Then the $A$-module $\Lat$ is generated by the subset 
\begin{equation*}
\Latb{\volu \cdot \econs^n}
\end{equation*}
with
$\econs = {\cco}^{\econx}$.
\end{thm}
\begin{proof}
Consider the locally free sheaf $\Bu = \Vgrb{1}$. We have
\begin{equation*}
\Ho^0(X,\:\Etwneg{\Bu}) = \Latb{\twc^{-1}}.
\end{equation*}
Thus $\Ho^0(X,\:\Etwneg{\Bu}) = 0$
by our assumption on the length of nonzero lattice vectors.
Invoking Riemann--Roch formula we deduce that
\begin{equation*}
-\deg(\Etwneg{\Bu}) = -\chi(\Bu) + (-g + 1 + \rdeg) n.
\end{equation*}
Hence by Lemma~\ref{cmrvan} the sheaf $\Etwn{\Bu}{(i-1)}$
is globally generated when
\begin{equation*}
i = \left\lceil\frac{-\chi(\Bu) + (g - 1 + \rdeg + \rdeg\vah)n + 1}{\rdeg}\right\rceil.
\end{equation*}
To ensure that the invertible sheaf $\Atwn{\vah}$
is very ample it is enough to take any integer $\vah$ such that $\rdeg\vah\geqslant 2g + 1$.
We thus have an estimate $\rdeg\vah\leqslant 2g + \rdeg$,
and the sheaf $\Etwn{\Bu}{j}$ is globally generated for an integer $j$ such that
\begin{equation*}
j \leqslant \left\lceil\frac{-\chi(\Bu) + (\econx)n + 1}{\rdeg}\right\rceil - 1.
\end{equation*}
In particular the $A$-module $\Lat = \Ho^0(Y,\:\Etwn{\Bu}{j})$ is generated by the subset
\begin{equation*}
\Latb{\twc^j} = \Ho^0(X,\:\Etwn{\Bu}{j}).
\end{equation*}
Since $\twc = {\cco}^{\rdeg}$ we obtain an estimate
\begin{equation*}
\twc^j \leqslant {\cco}^{-\chi(\Bu)} \cdot {\cco}^{(\econx)n}.
\end{equation*}
The result follows as ${\cco}^{-\chi(\Bu)} = \volu$ by Lemma~\ref{voleuler}.
\end{proof}

In the case $A=\fco[t]$ the estimate of Theorem~\ref{volbound} is sharp in all ranks.
The lattice $A^{\oplus n}$ with the supremum norm has volume ${\cco}^{-n}$,
is generated by vectors of norm~$1$ and has no nonzero vectors of norm
strictly less than~$1$.

%vspace*{-.3em}
\section{Preliminaries concerning Drinfeld modules}

Fix a field $\Fq$ of cardinality~$q < \infty$.
Let $K$ be a field over~$\Fq$.
In the following we will refer to $K$ as the \emph{base field}.
Pick a separable closure $K^s/K$ and let $G_K$ be the corresponding absolute Galois group.

\subsection{Twisted polynomials}

%\marked{Fix the form of Moore determinants somewhere?}
Let $K[\tau]$ be the \emph{twisted polynomial ring}.
This is the ring of polynomials in a formal variable $\tau$
with coefficients in $K$ and with the multiplication 
subject to the rule
$\tau\,\alpha = \alpha^q\,\tau$ for all $\alpha\in K$.
The elements of $K[\tau]$ will be referred to as
\emph{$\tau$-polynomials}.

The ring $K[\tau]$ is 
the endomorphism ring of the $\Fq$-module scheme $\Ga$ over
$\Spec K$ with the element $\tau$ corresponding to the $q$-Frobenius.
This interpretation allows us to evaluate $\tau$-polynomials
at points of $\Ga$.
% A \emph{root} of a $\tau$-polynomial $\varphi$ is an element
% $\alpha \in K^s = \Ga(K^s)$ which is mapped by~$\varphi$ to zero.
%
%\begin{dfn}
%A $\tau$-polynomial is \emph{separable} if its constant coefficient is not zero.
%\end{dfn}

% \begin{lem}\label{polyvanishideal}
% Let $\varphi \in K[\tau]$ be a separable polynomial and let~$V \subset K^s$
% be the $\Fq$-vector space of its roots.
% Then the set of polynomials $\psi \in K[\tau]$ which vanish on~$V$
% is the left ideal generated by $\varphi$.
% \end{lem}

\begin{lem}\label{addreal}
Let $V$ be a finite-dimensional Galois-stable $\Fq$-vector subspace of~$K^s$
and let $f\colon V \to K^s$ be a Galois-equivariant $\Fq$-linear map.
Set $n = \dim V$ and suppose that $n > 0$. Then the following holds:
\begin{statements}
\item\label{addrealeu}
There is a unique $\tau$-polynomial of the form
\begin{equation*}
\varphi = a_0 + a_1 \, \tau + \dotsc + a_{n-1} \, \tau^{n-1}, \quad a_i \in K,
\end{equation*}
which satisfies $\varphi(v) = f(v)$ for all $v \in V$.

\item\label{addrealexpr}
The coefficients $a_i$ are determined by the equation
\begin{equation*}%\label{moorematrix}\tag{$\dagger$}
\begin{pmatrix}
v_1 & v_1^q & \dotsc & v_1^{q^{n-1}} \\
v_2 & v_2^q & \dotsc & v_2^{q^{n-1}} \\
& & \dotsc & \\
v_n & v_n^q & \dotsc & v_n^{q^{n-1}}
\end{pmatrix}
\begin{pmatrix}
a_0 \\ a_1 \\ \vdots \\ a_{n-1}
\end{pmatrix} =
\begin{pmatrix}
f(v_1) \\ f(v_2) \\ \vdots \\ f(v_n)
\end{pmatrix}.
\end{equation*}
where $v_1, \dotsc, v_n$ is a basis of~$V$.
\end{statements}
\end{lem}
\begin{proof}
%Pick a basis $v_1, \dotsc, v_n$ of $V$.
The determinant of the square matrix in~\ref{addrealexpr}
is the Moore determinant $M(v_1,\dotsc,v_n)$,
see \cite[Definition 3.1.17]{papikian}.
This determinant is non-zero because the elements $v_1, \dotsc, v_n$ are $\Fq$-linearly independent.
Hence the solution vector %$(a_0, \,\dotsc, \,a_{n-1}) \in (K^s)^{\oplus n}$
is uniquely determined.

Consider the polynomial
%\begin{equation*}
$\varphi = a_0 + a_1 \, \tau + \dotsc + a_{n-1} \, \tau^{n-1}$.
%\end{equation*}
By construction we have
$\varphi(v_i) = f(v_i)$
for all $i \in \{1,\dotsc,n\}$ which implies by linearity
that $\varphi(v) = f(v)$ for all $v \in V$.
It remains to show that the coefficients of $\varphi$
are Galois-invariant and thus belong to the subfield $K\subset K^s$.

Apply an automorphism $g \in G_K$
to both sides of the matrix equation in~\ref{addrealexpr}.
Since the map $f$ is Galois-equivariant it follows that the polynomial
\begin{equation*}
{}^g \varphi = g(a_0) + g(a_1) \, \tau + \dotsc + g(a_{n-1}) \, \tau^{n-1}
\end{equation*}
satisfies $({}^g \varphi)(g v_i) = f(g v_i)$ for all $i$.
The set $g v_1, \dotsc, g v_n$ is yet another basis of~$V$
so the linearity implies that
$({}^g \varphi)(v) = f(v)$ for all $v \in V$.
In particular $({}^g \varphi)(v_i) = f(v_i)$ for all $i$.
Since the matrix equation in~\ref{addrealexpr} has a unique solution
we deduce that ${}^g \varphi = \varphi$, and the claim follows.
\end{proof}

\subsection{Artin--Schreier theory}\label{astheory}
\newcommand{\isog}{\varphi}
\newcommand{\isoga}{\psi_0}
\newcommand{\isogb}{\psi_1}
\newcommand{\kerisog}{V}

Let $E$ be an $\Fq$-module scheme over $\Spec K$ which is isomorphic to~$\Ga$.
Consider a \emph{separable isogeny} $\isog\colon E \to E$,
i.e. an $\Fq$-linear morphism that induces an isomorphism of the tangent spaces at zero.
Let $\kerisog$ be the kernel of $\isog$ viewed as a morphism of sheaves on the small \'etale site of~$\Spec K$.
The sheaf $\kerisog$ %is finite \'etale,
corresponds to the $\Fq$-vector space of roots
\begin{equation*}
\{ y \in E(K^s), \:\isog(y) = 0 \}
\end{equation*}
equipped with the natural Galois action. The short exact sequence of sheaves
\begin{equation*}
\begin{tikzcd}
0 \rar & \kerisog \rar & E \rar["\isog"] & E \rar & 0
\end{tikzcd}
\end{equation*}
induces a long exact sequence of cohomology that terminates in a surjective boundary homomorphism
\begin{equation*}
\begin{tikzcd}[column sep=3em]
E(K) \rar[two heads,"\AS_\varphi"] & \Ho^1(K,\:\kerisog).
\end{tikzcd}
\end{equation*}
The boundary homomorphism $\AS_\isog$ sends $x \in E(K)$ to the class of the cocycle
%begin{equation*}
$g \mapsto g y - y$
%end{equation*}
with $y \in E(K^s)$ any element satisfying $\isog(y) = x$.

% Consider a \emph{separable isogeny} $\isog\colon E \to E$, i.e.
% an $\Fq$-linear morphism which induces an isomorphism on Lie algebras.
% The kernel of $\isog$ is a finite \'etale $\Fq$-module scheme over $\Spec K$.
% 
% Such an object is the same thing as a finite-dimensional $\Fq$-vector space
% equipped with an action of the absolute Galois group~$G_K$.
% The scheme $\kerisog = \ker\isog$ corresponds to the $\Fq$-vector space of roots
% \begin{equation*}
% \{ y \in E(K^s) \mid \isog(y) = 0 \}
% \end{equation*}
% equipped with the induced Galois action.
% 
% Restricting to the small \'etale site of $\Spec K$ we obtain a short
% exact sequence of $\Fq$-module sheaves:
% \begin{equation*}
% \begin{tikzcd}
% 0 \rar & \kerisog \rar & E \rar["\isog"] & E \rar & 0.
% \end{tikzcd}
% \end{equation*}
% This induces an exact cohomology sequence:
% \begin{equation*}
% \begin{tikzcd}
% 0 \rar & \Ho^0(K,\:\kerisog) \rar & E(K) \rar["\isog"] \rar & E(K) \rar["\AS_\varphi"] & \Ho^1(K,\:\kerisog) \rar & 0.
% \end{tikzcd}
% \end{equation*}
% The boundary homomorphism $\AS_\isog$ sends $x \in E(K)$ to the class
% of the cocycle
% \begin{equation*}
% g \mapsto g y - y
% \end{equation*}
% with $y \in E(K^s)$ any element satisfying $\isog(y) = x$.

We will refer to $\AS_\isog$ as the \emph{Artin--Schreier homomorphism}
of the isogeny~$\isog$.
%\begin{dfn}
We will use the shorthand $\AS_q$ to denote
the Artin--Schreier homomorphism of the isogeny
$\tau - 1\colon \Ga \to \Ga$ with kernel the constant sheaf~$\Fq$.
%\end{dfn}

Suppose that the sheaf~$\kerisog$ is constant, or equivalently
that the $\Fq$-vector space $\kerisog$ is contained in $E(K)$. 
Then the cohomology group $H^1(K,\:\kerisog)$ coincides with the $\Fq$-vector space
of continuous homomorphisms
$G_K\to \kerisog$. In this case we will often write the Artin--Schreier
boundary homomorphism $\AS_\isog$
in the form of a pairing
\begin{equation*}
\kumn{\isog}\colon E(K) \times G_K \to \kerisog, \quad
\kumu{x}{\isog} = \AS_{\varphi}(x).
\end{equation*}
We will use the shorthand $\kumn{q}$ for the pairing of the isogeny $\tau-1$.%\colon \Ga \to \Ga$.

\begin{lem}\label{asfactor}
Let $\isog\colon \Ga \to \Ga$ be a separable isogeny
with kernel~$V$.
Then for every Galois-equivariant epimorphism $f\colon V \twoheadrightarrow \Fq$ 
there is a scaling factor $u \in K$ which makes the following
square commutative:
\begin{equation*}
\begin{tikzcd}[column sep=5em]
K \rar["x\:\mapsto\:u \, x"] \dar["\AS_\isog"'] & K \dar["\AS_q"] \\
\Ho^1(K,\:V) \rar["{\Ho^1(f)}"] & \Ho^1(K,\:\Fq)
% K \rar["\AS_V"] \dar["x \mapsto u x"'] & \Ho^1(K,\:V) \dar["\Ho^1(f)"] \\
% K \rar["\AS"] & \Ho^1(K,\:\Fq)
\end{tikzcd}
\end{equation*}
The factor $u$ is computed via Moore determinants:
\begin{equation*}
u = \alpha^{-1} \, \left(\frac{M(w_1,\,\dotsc\,,\,w_n)}{M(w_1,\,\dotsc\,,\,w_n,\,v)}\right)^q.
\end{equation*}
The terms in this formula have the following meaning:
\begin{itemize}
\item
$\alpha$ is the top coefficient of the $\tau$-polynomial $\isog$,

\item
$v \in V$ is an element such that $f(v) = 1$,

\item
$w_1, \dotsc, w_n \in V$ is a basis of the hyperplane $\ker(f) \subset V$.
\end{itemize}
The value of $u$ is independent of choices of
$v$ and $w_1,\dotsc,w_n$.
\end{lem}
The factor~$u$ of Lemma~\ref{asfactor} is unique provided that $\Ho^1(K,\,\Fp) \ne 0$.
\begin{proof}[Proof of Lemma~\ref{asfactor}]
As a first step we shall prolong the morphism $f\colon V \to \Fq$ to a morphism
of short exact sequences of sheaves:
\begin{equation*}
\begin{tikzcd}
0 \rar & V \dar["f"] \rar & \Ga \dar[dashed,"\isoga"] \rar["\isog"] & \Ga \dar[dashed,"\isogb"] \rar & 0 \\
0 \rar & \Fq \rar & \Ga \rar["\tau - 1"] & \Ga \rar & 0
\end{tikzcd}
\end{equation*}
By Lemma~\ref{addreal}~\ref{addrealeu} there is a unique polynomial 
$\isoga \in K[\tau]$ of degree strictly less than $\dim V = n + 1$ which makes
the left square commutative. The composite $(\tau - 1) \, \isoga$ vanishes
on the $\Fq$-vector subspace~$V \subset K^s$, and is thus divisible
by the polynomial $\isog$ on the right \cite[Corollary~3.1.16]{papikian}.
We get the existence of the polynomial $\isogb$.
By construction $\deg \isoga \leqslant n$ and $\deg \isog = n+1$ so the relation
\begin{equation*}
\isogb \, \isog = (\tau - 1)\,\isoga
\end{equation*}
implies that the polynomial $\isogb = u$ is constant, and that
\begin{equation*}
u\,\alpha = \beta^q
\end{equation*}
where $\beta$ is the top coefficient of the polynomial $\isoga$.
This coefficient is computed by Lemma~\ref{addreal}~\ref{addrealexpr}
which yields the Moore determinant formula for~$u$.
\end{proof}

\subsection{Drinfeld modules}\label{predm}
\newcommand{\carE}{\iota_E}
\newcommand{\ndual}{\fn^{-1}}

Fix a global function field $\F$ over~$\Fq$ and a place~$\infty$ of~$\F$.
As in~\S\ref{coeffring} this determines a \emph{coefficient ring}~$A$,
i.e. the subring of~$\F$ consisting of elements which are regular
outside~$\infty$. We will freely use the notation and the terminology of~\S\ref{coeffring}
with respect to coefficient rings and related objects.

For us a \emph{Drinfeld $A$-module} over $\Spec K$ is an $A$-module
scheme~$E$ such that the underlying $\Fq$-module scheme is isomorphic
to~$\Ga$, and the multiplication
morphism $a\colon E \to E$ is an isogeny of degree strictly greater than~$1$
for at least one element $a \in A$.

Subject to a choice of an $\Fq$-linear isomorphism $E \isoarrow \Ga$
the data of an $A$-module structure on~$E$ amounts to a morphism of $\Fq$-algebras
$\varphi\colon A \to K[\tau]$, %$K[\tau] = \End_{\Fq} \Ga$,
and one recovers the more common definition of a Drinfeld module.

Let $\Lie_E$ be the tangent space of~$E$ at zero.
This is a vector space of dimension~$1$ over the base field~$K$.
The action of $A$ on $\Lie_E$ determines a morphism
$\carE\colon A \to K$, the \emph{characteristic homomorphism}
of the Drinfeld module~$E$.
%
%\begin{dfn}
The \emph{characteristic} of~$E$ is the ideal $\ker(\carE) \subset A$.
%\end{dfn}

Let $\fn \subset A$ be a proper ideal which is not divisible by
the characteristic of~$E$.
The $\fn$-torsion of~$E$ is most naturally represented by the $A/\fn$-module
\begin{equation*}
E[\fn] = \Hom_A(A/\fn,\:E(K^s)).
\end{equation*}
We will use a modified definition which makes it easier to form inverse systems:

\begin{dfn}
$\ntor{E} = \Hom_A(\ndual/A,\:E(K^s))$.
\end{dfn}

The $A/\fn$-module $\ntor{E}$ carries
an action of the absolute Galois group $G_K$ inherited
from~$E(K^s)$. As the $A/\fn$-module $\ndual/A$ is free of rank~$1$
a choice of its generator gives rise to a $G_K$-equivariant isomorphism of
$A/\fn$-modules:
\begin{equation*}
\ntor{E} \isoarrow E[\fn]. %\quad E[\fn] = \Hom_A(A/\fn,\:E(K^s)).
\end{equation*}
% The ideal $\fn$ is not divisible by the characteristic of~$E$ by assumption,
% and so the $A/\fn$-modules $E[\fn]$ and $\ntor{E}$ are free of rank~$r$ \addref.

An inclusion of ideals $\fn'\subset\fn$ induces a surjection
$\xtor{E}{\fn'} \twoheadrightarrow \xtor{E}{\fn}$, and the resulting morphism
$A/\fn \otimes_{A/\fn'} \xtor{E}{\fn'} \isoarrow \xtor{E}{\fn}$
is an isomorphism. % provided that the ideal $\fn'$ is prime to the characteristic of~$E$.

\begin{lem}\label{torses}
%For each proper ideal $\fn\subset A$ which is not divisible by the characteristic of~$E$
We have a Galois-equivariant short exact sequence of $A$-modules
which is natural in the ideal~$\fn$ and the Drinfeld module~$E$:
\begin{equation*}%\label{torsespic}\tag{$\ast$}%
\begin{tikzcd}
0 \rar & \ntor{E} \rar & \Hom_A(\fn^{-1},\:E(K^s)) \rar & E(K^s) \rar & 0.
\end{tikzcd}
\end{equation*}
The first arrow is given by the composition with the quotient map
$\fn^{-1} \twoheadrightarrow \fn^{-1}/A$ and the second arrow is
the evaluation at $1 \in \fn^{-1}$.
\end{lem}
\begin{proof}
We only need to prove the exactness on the right.
Suppose that $\fn = (a)$ for some $a \in A$. 
The isogeny $a\colon E \to E$ is separable since the ideal~$\fn$
is not divisible by the characteristic of~$E$. Hence the morphism
$a\colon E(K^s) \twoheadrightarrow E(K^s)$ is surjective, and we get the claim.

In the general case the ideal $\fn^i$ is principal for all sufficiently divisible integers $i > 0$
since the Picard group of the scheme~$\Spec A$ is finite.
The claim results from the naturality of our sequence with respect to the ideal~$\fn$.
\end{proof}

% \begin{lem}
% Let $\fn'\subset\fn\subset A$ be an inclusion of proper ideals which are prime to
% the characteristic of~$E$.
% Then the natural morphism
% % $\xtor{E}{\fn'} \twoheadrightarrow \ntor{E}$ is surjective
% % and the resulting morphism
% $A/\fn \otimes_{A/\fn'} \xtor{E}{\fn'} \isoarrow \ntor{E}$
% is an isomorphism.
% \end{lem}
% \begin{proof}
% Write $\fn' = \fn \cdot \fn''$ for a suitable ideal $\fn''\subset A$.
% The natural morphism
% $\fn' \otimes_A E(K^s) \to \fn \otimes_A E(K^s)$ decomposes as
% \begin{equation*}
% \fn \otimes_A \big(\fn'' \otimes_A E(K^s) \twoheadrightarrow E(K^s)\big).
% \end{equation*}
% The ideal $\fn''$ is prime to the characteristic of~$E$,
% so the morphism in brackets is surjective by Lemma~\ref{torses}.
% The naturality of Sequence~\eqref{torsespic} implies
% by means of the Snake Lemma that the morphism
% $\xtor{E}{\fn'} \twoheadrightarrow \xtor{E}{\fn}$ is onto,
% so the induced morphism
% $A/\fn \otimes_{A/\fn'} \xtor{E}{\fn'} \isoarrow \ntor{E}$
% is surjective. As its source and target are
% free $A/\fn$-modules of the same rank we conclude that this
% morphism is an isomorphism.
% \end{proof}

\begin{dfn}
The \emph{Kummer boundary homomorphism}
\begin{equation*}
\AS_{\fn}\colon E(K) \to \Ho^1(K,\:\ntor{E})
\end{equation*}
is the boundary homomorphism arising from the sequence of Lemma~\ref{torses}.
When the $\fn$-torsion of~$E$ is rational we will write it in the form of a pairing:
\begin{equation*}
\kumn{\fn}\colon E(K) \times G_K \to \ntor{E}.
\end{equation*}
The map $\AS_{\fn}$ can be also called the Artin--Schreier homomorphism but
this name is less common in the literature.
\end{dfn}

\begin{lem}\label{dmasdesc}
The homomorphism $\AS_{\fn}$ has the following description:
\begin{statements}
\item
For each $x \in E(K)$ there is an $A$-module homomorphism
$f\colon \fn^{-1} \to E(K^s)$ such that $f(1) = x$.

\item
The cohomology class $\AS_{\fn}(x)$ is represented by a cocycle
given by the formula $g \mapsto g \circ f - f$
where the right hand side is seen as a morphism from the quotient
$\fn^{-1}/A$.
\end{statements}
\end{lem}
\begin{proof}
This follows from Lemma~\ref{torses} by elementary properties
of group cohomology.
\end{proof}

\begin{lem}\label{dmasprincipal}
For each proper principal ideal $(a) = \fa$ of the ring~$A$ which is not divisible by the characteristic of~$E$
we have a commutative diagram
\begin{equation*}
\begin{tikzcd}
& E(K) \arrow[dl,"\AS_{\fa}"'] \arrow[dr,"\AS_{a}"] & \\
\Ho^1(K,\:\ator{E}) \arrow[rr,"\Ho^1(\tori)","\cdisochar"'] & & \Ho^1(K,\:E[a])
\end{tikzcd}
\end{equation*}
where $\AS_{a}$ is the Artin--Schreier boundary homomorphism of the separable isogeny $a\colon E \to E$, and
$\tori\colon \ator{E} \isoarrow E[a]$ is an isomorphism defined by the formula $x\mapsto x(a^{-1})$.
\end{lem}
\begin{proof}
We have a commutative diagram:
\begin{equation*}
\begin{tikzcd}
0 \rar & \ator{E} \dar["\tori"'] \rar & \Hom_A(\fa^{-1},\:E(K^s)) \dar["f \mapsto f(a^{-1})"] \rar & E(K^s) \dar[equals] \rar & 0 \\
0 \rar & E[a] \rar & E(K^s) \rar["a"] & E(K^s) \rar & 0.
\end{tikzcd}
\end{equation*}
The claim follows by naturality of boundary homomorphisms.
\end{proof}

% \begin{dfn}
% Let $\fp\subset A$ be a prime different from the characteristic of~$E$.
% The \emph{$\fp$-adic Tate module} $T_\fp E$ is defined by the formula
% \begin{equation*}
% T_\fp E = \varprojlim\nolimits_{n > 0} \xtor{E}{\fp^n}.
% \end{equation*}
% \end{dfn}
% This module is free of rank~$r$ over the ring $A_\fp$,
% and it comes equipped with a $\fp$-adically
% continuous action of the absolute Galois group~$G_K$.

\subsection{Drinfeld modules over local fields}\label{reschar}

In this section our base field~$K$ is assumed to be local.
We denote by $\OK$ the ring of integers of~$K$ and by
$\mK$ the maximal ideal of~$\OK$.

Pick a coefficient ring $A$ over $\Fq$. Let $E$ be a Drinfeld $A$-module
over $\Spec K$ and let $\carE\colon A \to K$ be its characteristic
homomorphism.

\begin{dfn}
The Drinfeld module~$E$ has \emph{finite residual characteristic}
when $\carE(A) \subset \OK$. The \emph{residual characteristic}~$\fpres$ of~$E$ is the ideal
\begin{equation*}
\fpres := \carE^{-1}(\mK).
\end{equation*}
This ideal is maximal because the residue field $\OK/\mK$ is finite.
\end{dfn}

The analogous situation for abelian varieties is when the
base field~$K$ is a non-archimedean local field. The 
residual characteristic of an abelian variety is just the
residual characteristic of the field~$K$.

From now on we suppose that the Drinfeld module~$E$ has finite residual characteristic.
It will be convenient for us to assemble the $\fp$-adic Tate modules $T_\fp E$, $\fp\ne\fpres$,
into a single object.
\begin{dfn}
The ring of \emph{restricted integral adeles} is defined by the formula
\begin{equation*}
\Aad = \varprojlim\nolimits_{\fn} A/\fn
\end{equation*}
where $\fn \subset A$ runs over proper ideals that are prime to~$\fpres$.
\end{dfn}
We have a natural isomorphism $\Aad = \prod_{\fp\ne\fpres} A_\fp$
where $A_\fp$ is the $\fp$-adic completion of~$A$.
Compared with the ring of integral adeles of the function field~$\F$
the ring $\Aad$ does not include the factors at the places $\fpres$ and $\infty$.

\begin{dfn}
The \emph{restricted adelic Tate module} $\Tad E$ is defined by the formula
\begin{equation*}
\Tad E = \varprojlim\nolimits_{\fn} \ntor{E}
\end{equation*}
where $\fn \subset A$ runs over proper ideals that are prime to~$\fpres$ and the transition maps are induced
by inclusions of the ideals.
\end{dfn}
The Tate module $\Tad E$ is a free module over the ring $\Aad$
and its rank coincides with the rank of the Drinfeld module~$E$.
The Galois action on $\Tad E$ is $\Aad$-linear.
We have a natural $\Aad$-linear Galois-equivariant isomorphism:
\begin{equation*}
\Tad E \isoarrow \prod\nolimits_{\fp\ne\fpres} T_\fp E.
\end{equation*}%

\begin{lem}\label{initial}
Consider the inclusion-ordered family $\cF$ of proper principal ideals $\fa \subset A$
that are prime to the residual characteristic~$\fpres$.
Then the natural homomorphism
\begin{equation*}
\Tad E \isoarrow \varprojlim\nolimits_{\fa \in \cF} \xtor{E}{\fa}
\end{equation*}
is an isomorphism.
\end{lem}
\begin{proof}
For each proper ideal $\fn\not\subset\fpres$ there is
a non-constant element $a \in \fn \setminus\fpres$.
Then $(a) \in \cF$, and as $(a) \subset \fn$ we conclude that the family
$\cF$ is initial in the family of all proper ideals not contained in $\fpres$. 
The claim follows.
\end{proof}

\section{The local Kummer pairing}\label{locmonpair}

Fix a finite field $\Fq$ of cardinality~$q$
and characteristic~$p$.
In this section our base field $K/\Fq$ is assumed to be local.
We denote by $\OK\subset K$ the ring of integers and
by $\mK \subset \OK$ the maximal ideal.
As usual we fix a separable closure $K^s/K$ and denote by
$G_K$ the corresponding absolute Galois group.

\subsection{Ramification subgroups}\label{ramsubgroups}

%From now on we assume that the base field $K/\Fq$ is local.
Let $I_K$ be the inertia subgroup of the absolute
Galois group $G_K$, and let $P_K$ be the wild inertia subgroup.

\begin{dfn}
We denote by $J_K$ the maximal quotient of $I_K$ that is abelian of exponent $p$:
\begin{equation*}
J_K := I_K^\ab/(I_K^\ab)^{\times p}.
\end{equation*}
\end{dfn}

The quotient group $I_K/P_K$ is procyclic of order prime to $p$,
so the fact that $P_K$ is a pro-$p$ group implies that
the surjection $I_K \twoheadrightarrow I_K/P_K$ is split. %because $P_K$ is a pro-$p$ group.
As the group $J_K$ is abelian $p$-torsion we conclude that the canonical morphism
$P_K \twoheadrightarrow J_K$ is surjective.

\medskip
% \begin{lem}\label{wildsurj}
% Let $P_K \subset I_K$ be the wild inertia subgroup.
% Then the natural morphism $P_K \twoheadrightarrow J_K$ is surjective.
% \end{lem}
% \begin{proof}
% The tame inertia group $T_K = I_K/P_K$ is procyclic of order prime to~$p$.
% As $P_K$ is a pro-$p$ group we deduce that
% $I_K \simeq P_K \rtimes T_K$, and the claim follows from the fact that
% the group $J_K$ is $p$-torsion.
% \end{proof}

% = \Gal(K^s/K)$.
%
%\marked{Inertia subgroup $I_K$ and its ramification filtration.}
The inertia subgroup $I_K$ carries a decreasing separated exhaustive filtration by closed normal subgroups $I_K^{\ram{u}}$,
$u \in \Q_{\geqslant 0}$,
the \emph{ramification filtration in upper numbering} %, see
\cite[Chapitre~IV, \S3, Remarque~1, p.~83]{corps}.
%We have $I_K = I_K^0$ and $\bigcap_{u\geqslant 0} I_K^{\ram{u}} = \{1\}$.
The wild inertia subgroup $P_K$ is the closure of the subgroup
$\bigcup_{u > 0} I_K^{\ram{u}}$.
%\marked{Relate to wild inertia?}
% % % The filtration is continuous on the left in the sense
% % % that for each $u\geqslant 0$ we have an equality
% % % \begin{equation*}
% % % I_K^u = \bigcap_{\mu < u} I_K^\mu.
% % % \end{equation*}
% For each $u\geqslant 0$ the subgroup
% \begin{equation*}
% I_K^{\ramx{u}} = \bigcup\nolimits_{\mu > u} I_K^{\ram{\mu}}.
% \end{equation*}
% is closed.
% With this notation the wild inertia subgroup is $I_K^{\ramx{0}}$.

The ramification filtration of $I_K$ induces a filtration of $J_K$.
As the group $J_K$ is abelian the Hasse--Arf theorem implies that the induced filtration is integrally indexed:
% in the following way:
% Ramification subgroups
% $I_K^{\ram{u}}$ map surjectively onto $J_K^{\ram{n}}$ with $n = \lceil u \rceil$.
For each $u\geqslant 0$ the subgroup $J_K^u$ coincides with $J_K^\coi$, $\coi = \lceil u \rceil$.
%and the subgroups $\smash{I_K^{\ramx{n}}}$ with integer indices $n$ map onto $\smash{J_K^{\ram{n+1}}}$.

% We have $\smash{J_K^{\ram{0}}} = \smash{J_K^{\ram{1}}}$ since the tame inertia group
% $I_K/\smash{I_K^{\ramx{0}}}$ is procyclic of order prime to~$p$.

\begin{lem}\label{ramcong}
For every integer $\coi \geqslant 0$ such that $\coi \equiv 0\pmod{p}$ we have %  an equality:
\begin{equation*}
J_K^\coi = J_K^{\coi+1}.
\end{equation*}
\end{lem}
%Note that Lemma~\ref{wildsurj} amounts to the special case $\coi = 0$ of Lemma~\ref{ramcong}.
\begin{proof}%[Proof of Lemma~\ref{ramcong}]
Let $L\subset K^s$ be a finite Galois extension of~$K$
with Galois group~$G$ and upper index ramification filtration
$\{ G^u \}_{u\geqslant 0}$. Suppose that the inertia group $G^0$
is abelian $p$-torsion. We need to show that $G^{pi} = G^{pi+1}$
for each $i \geqslant 0$.

We are free to assume that $\smash{G^0} = G$.
The group~$G$ is abelian, so we have the reciprocity morphism
$\omega\colon K^\times \twoheadrightarrow G$ of local class field theory.
This morphism transforms the $n$-unit filtration $\smash{\{U_K^{(n)}\}_{n\geqslant 0}}$
%on the multiplicative group $K^\times$
to the ramification filtration.
For each $i \geqslant 0$ the $p$-th power map induces a morphism
\begin{equation*}
U_K^{(i)} \twoheadrightarrow U_K^{(pi)}/U_K^{(pi+1)}.
\end{equation*}
This is surjective since the residue field of~$K$ is perfect.
Our claim follows.
\end{proof}

%\begin{cor}\label{wildsurj}
%The natural morphism $P_K \twoheadrightarrow J_K$ is surjective.
%\end{cor}
%\begin{proof}
%Indeed the image of $P_K$ is the ramification subgroup $J_K^1$,
%and $J_K^1 = J_K^0$ by Lemma~\ref{ramcong}.
%\end{proof}

\subsection{Local pairing}\label{mainest}

Fix a Drinfeld $A$-module $\Eg$ over $\Spec \OK$.
By this we mean not only that $\Eg$ can be defined
by a homomorphism $\varphi\colon A \to \OK[\tau]$ but also that
the reduction of $\varphi$ modulo $\mK$ is a Drinfeld module
of the same rank as over $K$.

The Drinfeld module $\Eg$ has finite residual characteristic,
denoted by $\fpres$ as usual.
Let $K^\unr\subset K^s$ be the maximal unramified extension
and let $\fn\subset A$ be a proper ideal not divisible by $\fpres$.
Since $\Eg$ is a Drinfeld module over $\Spec\OK$ the
$\fn$-torsion of $\Eg$ is contained in $\Eg(K^\unr)$.
Hence by \S\ref{predm} we have the Kummer pairing
%\begin{equation*}
$\kumn{\fn}\colon \Eg(K^\unr) \times I_K \to \ntor{\Eg}$.
%\end{equation*}
% %
% In view of Lemma~\ref{asystem} such pairings are compatible with inclusions $\fa\subset\fa'$.
% Lemma~\ref{initial}
% implies that taking the limit over proper principal ideals not divisible by the residual characteristic $\fpres$
% we obtain a pairing
% \begin{equation*}
% \kumad{\:\:}{\:}\colon \Eg(K^\unr) \times I_K \to \Tad \Eg.
% \end{equation*}
% 
% We call this the \emph{local adelic Artin--Schreier pairing} of the Drinfeld module~$\Eg$.
\begin{dfn}\label{dfnlocpair}
The \emph{local restricted adelic Kummer pairing} %of the Drinfeld module~$\Eg$
\begin{equation*}
\kumadv\colon \Eg(K^\unr) \times I_K \to \Tad \Eg
\end{equation*}
is the limit of the Kummer pairings $\kumn{\fn}$ over 
the inclusion-ordered system of
proper ideals $\fn$ not divisible by the residual characteristic $\fpres$.
\end{dfn}
% The individual pairings $\kumn{\fa}$ are compatible with transition morphisms
% by Lemma~\ref{asystem}, and the limit gives the adelic Tate module $\Tad\Eg$
% by Lemma~\ref{initial}.
Since the Tate module $\Tad\Eg$ is $p$-torsion
the pairing $\kumadv$ factors through the maximal quotient $J_K = I_K^\ab/(I_K^\ab)^{\times p}$
that is abelian of exponent~$p$.
%We shall study the homomorphisms which arise by fixing the first variable of the local adelic pairing.

%\subsection{The main estimate}\label{mainest}

Let $v\colon K \twoheadrightarrow \Z \cup\{\infty\}$ be the normalized
valuation of the base field~$K$.
Every isomorphism of $\Fq$-module schemes $\Eg \isoarrow \Ga$ over $\Spec \OK$
induces
a valuation function on $\Eg(K)$ by transport of structure from $\Ga(K) = K$.
The result is independent of choices %of the isomorphism
because the choices differ by multiplication by an integral unit on the side of $\Ga$.
% Thus the normalized valuation $v$ induces
% a canonical valuation function on $\Eg(K)$.
We denote the resulting valuation function by~$v$
as well.

\begin{dfn}\label{dfnseminorm}
For every $\lv \in \Eg(K)$ we set
\begin{equation*}
\infn{\lv} = \max\{0, \:-v(\lv)\}.
\end{equation*}
%where the exponent $s$ is the rank of the Drinfeld module~$\Eg$.
\end{dfn}
Poonen~\cite{poonen} introduced the notion of canonical local height on
the $A$-module $\Eg(K)$.
Since the Drinfeld module~$\Eg$
has good reduction
Proposition~4~(4) of \cite[\S3]{poonen} implies that
the map $\infn{\cdot}\colon \Eg(K) \to \Z_{\geqslant 0}$ is exactly the canonical local height.
This map satisfies the ultrametric inequality:
\begin{equation*}
\infn{\lv + \lv'} \leqslant \max\{ \infn{\lv}, \:\infn{\lv'} \}.
\end{equation*}
However $\infn{\lv} = 0$ for every $\lv\in \Eg(\OK)$
so the height $\infn{\cdot}$ is only a seminorm on $\Eg(K)$.
This seminorm is $\infty$-adic in the following sense:

\begin{lem}\label{valrel}
For each $\lv\in \Eg(K)$
and each $a \in A$ we have an equality
\begin{equation*}
\infn{a\lv} = |a|_\infty^s \, \infn{\lv}
\end{equation*}
where $s$ is the rank of the Drinfeld module~$\Eg$.
\end{lem}
The map $x \mapsto |x|_\infty^s$ is an $\infty$-adic
absolute value on the local field $\Finf$ albeit its normalization
differs from the one of \S\ref{coeffring} when $s > 1$.

\begin{proof}[Proof of Lemma~\ref{valrel}]
This follows from Proposition~2 and Proposition~4~(4) of~\cite{poonen}.
For the reader's convenience let us give a direct argument.
The claim is clear when $a = 0$ and holds when $\infn{\lv} = 0$ since
$\Eg(\OK)$ is an $A$-submodule of~$\Eg(K)$. We are thus free to assume that
$a \ne 0$ and $\infn{\lv} > 0$.

We identify the $\Fq$-module scheme $\Eg$ with $\Ga$ over $\Spec \OK$.
The induced $A$-module scheme structure is described by a homomorphism
$\varphi\colon A \to \OK[\tau]$.
Write $\varphi(a) = \alpha_0 + \alpha_1 \, \tau + \dotsc + \alpha_n \, \tau^n$
with $\alpha_n \ne 0$. The top coefficient
$\alpha_n$ is a unit since $\Eg$ is a Drinfeld module over
$\Spec \OK$. % (in the sense specified at the beginning of the section).
Hence 
%begin{equation*}
$v(\alpha_n \, \lv^{q^n}) = q^n \, v(\lv)$.
%end{equation*}
Combining this with the estimates
$v(\alpha_i \, \lv^{q^i}) \geqslant q^i \, v(\lv)$, $i < n$,
and the assumption
$v(\lv) < 0$ we deduce an equality
\begin{equation*}
v(a \, \lv) = q^n \, v(\lv).
\end{equation*}
The degree $n$ of the polynomial $\varphi(a)$ is expressed by a formula
\begin{equation*}
n = -s \, [\kinf:\Fq] \, v_\infty(a)
\end{equation*}
where
$v_\infty$ is the normalized valuation of the local field $\Finf$.
The absolute value of \S\ref{coeffring} is given by the formula
\begin{equation*}
|a|_\infty = |\kinf|^{-v_\infty(a)}.
\end{equation*}
As $|\kinf| = q^{[\kinf:\Fq]}$ we deduce that
$q^n = |a|_\infty^s$ and the claim follows.
\end{proof}

% The canonical local height $\infs{\cdot}$ is integer-valued, and the arithmetic
% properties of its values will be important for us. At the same time it is the seminorm
% $\infn{\cdot}$ which fits well with the theory of \S\ref{lattices}.

\newcommand{\KH}{}%Artin--Schreier}
\newcommand{\ind}{n}
\begin{thm}\label{kumimage}
Let $\lv \in \Eg(K)$ be an element,
and consider the homomorphism
\begin{equation*}
\kumadl\colon J_K \to \Tad \Eg.
\end{equation*}
Then the following holds:
\begin{statements}
\item\label{rambound}
This homomorphism vanishes on the ramification subgroup $J^\ram{\infn{\lv}+1}_K$.
In particular it vanishes altogether when $\infn{\lv} = 0$.

\item\label{primetop}
%Suppose that
If the height $\infn{\lv}$ is prime to~$p$ %strictly positive and prime to~$p$.
then the homomorphism $\kumadl$ maps the ramification subgroup $\smash{J^\ram{\infn{\lv}}_K}$
surjectively onto~$\Tad \Eg$. %and is in particular surjective.
\end{statements}
\end{thm}
\begin{proof}
By Lemma~\ref{initial} and Lemma~\ref{dmasprincipal}
it is enough to prove the claims for every \KH homomorphism
$\kuml{a}\colon J_K \to \Eg[a]$ where $a \in A$
runs over non-constant elements that are prime to the residual characteristic~$\fpres$ of~$\Eg$.
 
The height $\infn{\lv}$,
the group $J_K$ and the \KH homomorphism $\kuml{a}$
remain unchanged when we replace the base field $K$ by an unramified extension.
As the Drinfeld module $\Eg$ has good reduction and the element $a$ is prime to~$\fpres$
we are free to assume that the torsion module
$\Eg[a]$ is contained in $\Eg(K)$, and consequently in $\Eg(\OK)$.

An $\Fp$-vector subspace $V\subset \Eg[a]$ is zero
if and only if $f(V) = 0$ for each nonzero linear form
$f\colon \Eg[a] \twoheadrightarrow \Fp$.
Likewise $V = \Eg[a]$
if and only if $f(V) = \Fp$ for every such form~$f$.
It is thus enough to prove our claims after composing the \KH homomorphism
$\kuml{a}$ with every $f$.

From now on we identify the $\Fq$-module scheme $\Eg$ with $\Ga$ over $\Spec \OK$.
The resulting $A$-module scheme structure on $\Ga$ is given by a homomorphism
$\varphi\colon A \to \OK[\tau]$.
%\marked{Clarify why $\kuml{a} = \kuml{\varphi(a)}$.}

Let $f\colon \Eg[a] \twoheadrightarrow \Fp$ be a nonzero linear form.
According to Lemma~\ref{asfactor} there is an element $u \in \Ga(K)$ such that
\begin{equation*}
f \circ \kuml{\varphi(a)} = \kumu{u\lv}{p}.
\end{equation*}
This element is given by a formula with Moore determinants $M(-)$:
\begin{equation*}
u = \alpha^{-1}\:\Big( \frac{M(w_1,\,\dotsc,\,w_\ind)}{M(w_1,\,\dotsc,\,w_\ind,\,v)}\Big)^p
\end{equation*}
Here $\alpha$ is the top coefficient of the polynomial $\varphi(a)$ and
$w_1, \dotsc,\,w_\ind,\,v$ is a suitable $\Fp$-basis of $\Eg[a]$.

By our assumption $\Eg$ is a Drinfeld module over $\Spec\OK$.
This implies that the top coefficient $\alpha$ is an integral unit.
Since $a$ is prime to the residual characteristic %$\fpres$
the torsion module $\Eg[a] \subset \Eg(\OK)$ maps injectively to $\Eg(k)$ where $k$
is the residue field of $\OK$.
In particular the image of the basis $w_1,\dotsc,w_\ind,v$ in $\Eg(k)$
remains $\Fp$-linearly independent.
Hence the corresponding
Moore determinants are nonzero.
We conclude that the element $u$ is an
integral unit. It follows that $v(u\lv) = v(\lv)$.

Consider the Artin--Schreier polynomial $X^p - X = u\lv$.
The Galois group $G$ of its splitting field is the image of the absolute Galois group %$G_K$
under the homomorphism $\kumu{u\lv}{p}$.
The upper index ramification filtration of $G$ is given by images of absolute ramification subgroups:
\begin{equation*}
G^{\ram{i}} = \kum{u\lv}{I_K^{\ram{i}}}{p}.
\end{equation*}
This filtration was calculated by Hasse, see Thomas~\cite{thomas},
Proposition~2.1 and the following paragraph.
Set $b = \max \{0,\,-v(u\lv)\}$. 
Then the subgroups $G^{\ram{i}}$ vanish for $i > b$,
and if the integer $b$ is prime to~$p$ %  strictly positive
then $G^{\ram{b}} = G = \Fp$.
Our theorem follows since $v(u\lv) = v(\lv)$ as shown above.
\end{proof}

\subsection{Open image theorem}

The following arguments involve passage from the base field~$K$ to a finite separable extension~$L$.
On account of this we have to fix a few conventions. First, it will be more convenient
for us to view the \KH homomorphism $\kumadl$ as a morphism
from the inertia subgroup $I_K$ rather than its quotient $J_K$.

Second, we need to be careful with the canonical local height
on the $A$-module $\Eg(L)$. This involves the normalized valuation of~$L$
and so differs from the height on 
the submodule $\Eg(K)$
by an integer factor, the ramification index of $L/K$. We will
denote the height on $\Eg(L)$ by $\infnx{L}{\cdot}$.

\begin{lem}\label{division}
Let $\lv \in \Eg(K) \setminus \Eg(\OK)$ be an element.
Suppose that the homomorphism
%\begin{equation*}
$\kumadl\colon I_K \to \Tad \Eg$
%\end{equation*}
is \emph{not} surjective. 
Then there are
\begin{itemize}
\item
a non-constant element $a \in A$ prime to the residual characteristic $\fpres$ of~$\Eg$,
%\setminus\fpres$, \marked{Verify that $\fpres$ is understood.}

\item
a finite separable extension $L/K$, and

\item
an element $\lv' \in \Eg(L)\setminus \Eg(\cO_L)$
\end{itemize}
such that
\begin{equation*}
a \lv' = \lv, \quad
v_p(\infnx{L}{\lv'}) < v_p(\infnx{K}{\lv}).
\end{equation*}
Here
$v_p\colon \Q \twoheadrightarrow \Z\cup\{\infty\}$ denotes the normalized $p$-adic valuation.
%and $w\colon \Eg(L) \to \Z\cup\{\infty\}$ is the normalized valuation.
\end{lem}
\begin{proof}
Lemma~\ref{initial} and Lemma~\ref{dmasprincipal}
imply that
there is a non-constant element $a \not\in\fpres$ such that the \KH homomorphism
$\kuml{a}\colon I_K \to \Eg[a]$ is not surjective.

The height $\infn{\lv}$,
the inertia group $I_K$ and the \KH homomorphism $\kuml{a}$
remain unchanged when we replace the base field $K$ by an unramified extension.
We are thus free to assume that the torsion scheme $\Eg[a]$ is constant, % over $\Spec\OK$,
or in other words, that all the $K^s$-points of $\Eg[a]$ are defined over the field $K$.
%
% In this case the \KH homomorphism is the restriction of the Artin--Schreier homomorphism
% $[\lv,\,-)_{a}\colon G_K \to \Eg[a]$, see \S\ref{astheory}. % arising from the Artin--Schreier sequence of $\Eg[a]$.
Under this assumption the \KH homomorphism $\kuml{a}$ is defined on the absolute Galois group $G_K$,
see \S\ref{astheory}.
Its kernel determines a finite abelian $p$-torsion extension $L/K$.

Let $I \subset \Eg[a]$ be the image of the inertia subgroup~$I_K$.
The ramification index $e$ of the extension $L/K$ equals the cardinality of $I$.
By our choice of the absolute value $|\cdot|_\infty$ in \S\ref{coeffring}
we have an equality
$|A/(a)| = |a|_\infty$.
So the torsion module $\Eg[a]$ has cardinality
$|a|_\infty^s$, and the assumption $I \ne \Eg[a]$ implies that
$e < |a|_\infty^s$.

By construction there is an element $\lv' \in \Eg(L)$ such that
$a \, \lv' = \lv$.
We have  $\infnx{L}{\lv} = e \, \infnx{K}{\lv}$
so Lemma~\ref{valrel} shows that
\begin{equation*}
|a|_\infty^s \, \infnx{L}{\lv'} = e \, \infnx{K}{\lv}.
\end{equation*}
The integers $e$ and $|a|_\infty$ are powers of~$p$
so the strict inequality $e < |a|_\infty^s$
implies a strict inequality of $p$-adic valuations
$v_p(\infnx{L}{\lv'}) < v_p(\infnx{K}{\lv})$
as claimed. %\marked{Check notation and exposition.}
\end{proof}

\pagebreak
\begin{thm}\label{kummeropen}
For each element $\lv \in \Eg(K) \setminus \Eg(\OK)$ the image of the homomorphism
$\kumadl\colon I_K \to \Tad \Eg$
is open. 
\end{thm}
\begin{proof}
Suppose that the \KH homomorphism is not surjective.
Invoking Lemma~\ref{division} we obtain a non-constant element $a \in A\setminus\fpres$,
a finite separable extension $L/K$
and an element $\lv' \in \Eg(L)$ such that
%\begin{equation*}
$a \, \lv' = \lv$. 
%\end{equation*}
%
As the pairing $\kumadv$ is $A$-linear in the first argument we deduce that
\begin{equation*}
a \kumad{\lv'}{I_L} = \kumad{\lv}{I_L}.
\end{equation*}
The right hand side has finite index in $\kumad{\lv}{I_K}$. %since the index of $I_L$ in $I_K$ is finite.
Hence it is enough to prove the theorem with $\lv$
replaced by $\lv'$ and $K$ replaced by~$L$.

Lemma~\ref{division} also shows that we have a strict inequality of $p$-adic
valuations:
\begin{equation*}
v_p(\infnx{L}{\lv'}) < v_p(\infnx{K}{\lv}).
\end{equation*}
The height $\infnx{L}{\cdot}$ takes only integer values,
so the valuation $v_p(\infnx{L}{\lv'})$ is bounded below by zero.
Hence after repeating our argument finitely many times we arrive
at a situation where the \KH homomorphism $\kumadu{\lv'}$
is onto.
\end{proof}

\section{Applications}

We retain the notation and the conventions of~\S\ref{locmonpair}.
In particular our base field~$K$ is a local field over
a finite field $\Fq$ of cardinality~$q$ and characteristic~$p$.
%We denote the ring of integers of $K$ by~$\OK$.
We fix a separable closure $K^s/K$, denote by $G_K$ the corresponding
absolute Galois group, by $I_K$ the inertia subgroup and by $J_K$ the
maximal quotient of $I_K$ that is abelian of exponent $p$, see \S\ref{ramsubgroups}

% Let $\Fq$ be a finite field of cardinality~$q$ and characteristic~$p$.
% Let $K$ be a local field over $\Fq$. As before we will refer to $K$
% as the \emph{base field}. We denote the ring of integers of $K$ by~$\OK$.
% We also fix a separable closure $K^s/K$ and denote by $G_K$ the corresponding
% absolute Galois group.
% 
% \marked{Add $I_K$, $J_K$.}

\subsection{The conductor of a Drinfeld module}\label{conductor}

Let $A$ be a coefficient ring over~$\Fq$ and
let $E$ be a Drinfeld $A$-module of stable reduction over $\Spec K$.
Drinfeld's theory of Tate uniformization~\cite[\S7]{drinfeld-ell}
represents $E$ as an analytic quotient of a Drinfeld module $\Eg$
over $\Spec\OK$ by a Galois-invariant \emph{lattice} $\Lat\subset \Eg(K^s)$,
a finitely generated projective $A$-submodule which is discrete with respect to the adic topology.
We thus have an analytic morphism
$e_\Lat\colon \Eg \to E$ extending to a short exact sequence:%the \emph{Tate uniformization sequence}:
\begin{equation*}%\label{tateseq}%\tag{$\ast$}
\begin{tikzcd}
0 \rar & \Lat \rar & \Eg \rar["e_\Lat"] & E \rar & 0.
\end{tikzcd}
\end{equation*}
This sequence is determined functorially by the Drinfeld module~$E$.
We will refer to $\Lat$ as the \emph{period lattice} of~$E$
and to $e_\Lat$ as the \emph{exponential map} of~$\Lat$.
A~detailed exposition of Tate uniformization theory can be found in a recent book of Papikian \cite[\S6.2]{papikian}.

Throughout this section we make an additional assumption: 
The lattice $\Lat$ is \emph{defined over~$K$}, which is to say, $\Lat\subset\Eg(K)$.
We will study the action of inertia on the Tate modules of~$E$.

\begin{lem}\label{quotchar}
The Drinfeld modules $\Eg$ and $E$ have the same characteristic homomorphism.
\end{lem}
\begin{proof}
Follows since the exponential $e_\Lat$ induces an isomorphism
of tangent spaces $\operatorname{Lie}_\Eg \isoarrow \operatorname{Lie}_E$.
\end{proof}

%\marked{Introduce monodromy representation.}
\newcommand{\liftx}{\tilde{x}}%
\newcommand{\delx}[1]{\delta_{#1}}%
\newcommand{\deln}{\delx{\fn}}%
\newcommand{\dela}{\delx{\fa}}%
\newcommand{\delad}{\delx{\ad}}%
\begin{lem}\label{modunip}%
For each proper ideal $\fn\subset A$ that is not divisible by the characteristic of~$E$
we have a Galois-equivariant short exact sequence of $A/\fn$-modules:
\begin{equation*}
\begin{tikzcd}[column sep=3.5em]
0 \rar & \ntor{\Eg} \rar["e_\Lat\circ-"] & \ntor{E} \rar["\deln"] & \Lat/\fn\Lat \rar & 0.
\end{tikzcd}
\end{equation*}
The homomorphism $\deln$ has the following description.
\begin{statements}
\item
For each $x \in \ntor{E}$ there is an $A$-module morphism
$\liftx\colon \fn^{-1} \to \Eg(K^s)$ which lifts~$x$
in the sense that $e_\Lat \circ \liftx$ equals the composite of $x$
and the quotient morphism $\fn^{-1} \twoheadrightarrow \fn^{-1}/A$.

\item
The residue class $\deln(x)$ is represented by the element
$\liftx(1) \in \Lat$ where $\liftx$ is a lift of~$x$.

\item\label{moduniprep}
Given a representative $\lv$ of the residue class $\deln(x)$ one can pick a lift
$\liftx$ such that $\liftx(1) = \lv$.
\end{statements}
\end{lem}
\begin{proof}
Consider a commutative diagram
\begin{equation*}
\begin{tikzcd}[column sep=3.5em]
0 \rar & \fn\Lat \dar \rar & \fn \otimes_A \Eg(K^s) \dar \rar["1\otimes e_\Lat"] & \fn\otimes_A E(K^s) \dar \rar & 0 \\
0 \rar & \Lat \rar & \Eg(K^s) \rar["e_\Lat"] & E(K^s) \rar & 0
\end{tikzcd}
\end{equation*}
with the vertical arrows given by $A$-module multiplication.
The bottom row of this diagram is exact by construction and the top
row is exact since the $A$-module $\fn$ is flat.

Consider the natural isomorphism 
$\fn\otimes_A M \isoarrow \Hom_A(\fn^{-1},\:M)$.
Under this isomorphism the multiplication map
$\fn \otimes_A M \to M$ is transformed to the evaluation
map $f \mapsto f(1)$.
Hence the middle vertical arrow in the diagram above is surjective by Lemma~\ref{torses},
and our claims follow from Snake Lemma.
\end{proof}

Lemma~\ref{quotchar} implies that Drinfeld modules $\Eg$ and $E$ have the same residual characteristic.
As usual we denote this characteristic by~$\fpres$.
Consider the restricted adelic Tate module $\Tad E$
and let $\GL(\Tad E)$ be the group of its automorphisms as an $\Aad$-module.
The action of the inertia subgroup $I_K$ % \subset G_K$ %on $\Tad E$
gives rise to the (local) \emph{restricted adelic monodromy representation}
\begin{equation*}
\rhad\colon I_K \to \GL(\Tad E).
\end{equation*}
Similarly for each prime $\fp\ne\fpres$ we have the \emph{$\fp$-adic monodromy representation}
%\begin{equation*}
$\rhp\colon I_K \to \GL(T_\fp E)$.
%\end{equation*}
The representation $\rhp$ is the composite of $\rhad$
and the projection $\GL(\Tad E) \to \GL(T_\fp E)$.
%induced by the projection $\Tad E \twoheadrightarrow T_\fp E$.

% By Lemma~\ref{initial} adelic Tate modules
% are limits of the respective torsion modules
% over the family of proper principal ideals which are prime to
% the residual characteristic. %~$\fpres$.
Taking the limit of sequences of Lemma~\ref{modunip}
we obtain a short exact sequence of Galois representations:
\begin{equation*}
\begin{tikzcd}[column sep=4em]
0 \rar & \Tad\Eg \rar["\Tad(e_\Lat)"] & \Tad E \rar["\delad"] & \Aad \otimes_A \Lat \rar & 0.
\end{tikzcd}
\end{equation*}
As the Drinfeld module~$\Eg$ has good reduction
we have at our disposal the local adelic Kummer pairing
\begin{equation*}
\kumadv\colon \Eg(K) \times I_K \to \Tad\Eg.
\end{equation*}
%see~Definition~\ref{dfnlocpair}. 
%Together with the period lattice $\Lat$ this pairing
%determines the monodromy representation:
%This pairing gives an explicit description of the monodromy representation:

\begin{lem}\label{unimon}
The monodromy representation $\rhad$ factors as a composition
\begin{equation*}
\begin{tikzcd}
I_K \arrow[rrr,"g\:\mapsto\:\kumad{\ }{g}"] &&& \Hom_A(\Lat,\:\Tad\Eg) \rar[hook] & \GL(\Tad E)
\end{tikzcd}
\end{equation*}
where the second arrow is given by the formula
$f \mapsto 1 + \Tad(e_\Lat) \circ f \circ \delad$.
\end{lem}
\begin{proof}
Let $\fn \subset A$ be a proper ideal not divisible by~$\fpres$
and let $x \in \ntor{E}$ be an element.
Pick a representative $\lv$ of the residue class $\deln(x)$.
We need to show that for each automorphism $g \in I_K$ one has
\begin{equation*}
g \circ x = x + e_\Lat \circ \kum{\lv}{g}{\fn}.
\end{equation*}
Let $\liftx\colon \fn^{-1} \to \Eg(K^s)$ be a lift of~$x$.
By Lemma~\ref{modunip}~\ref{moduniprep}
we are free to assume that $\liftx(1) = \lv$. 
Lemma~\ref{dmasdesc} shows that
\begin{equation*}
\kum{\lv}{g}{\fn} = g \circ \liftx - \liftx.
\end{equation*}
Hence
%\begin{equation*}
$e_\Lat\circ \kum{\lv}{g}{\fn} = g \circ x - x$,
%\end{equation*}
and we are done.
\end{proof}

\begin{cor}\label{abmon}
The monodromy representation $\rhad$ factors through the maximal
quotient $J_K = I_K^{\ab} / (I_K^{\ab})^{\times p}$ that is abelian of exponent~$p$.\qed
\end{cor}

\begin{rmk}
Similar reasnoning applies to abelian varieties over a nonarchimedean local field $K$.
Raynaud proved that
\begin{itemize}
\item
every abelian variety over $K$ has potentially semistable reduction
\cite[\S3]{grothendieck},

\item
an abelian variety $E$ of semistable reduction is a rigid analytic quotient of a semiabelian variety $D$ of good reduction by a lattice $\Lat$ \cite[\S14]{grothendieck}.
\end{itemize}
The Galois action on $\Lat$ is always unramified \cite[\S3, Cor. 3.8.]{grothendieck}.
Consequently, for each prime $\ell$ different from the residual characteristic the $\ell$-adic monodromy representation
$\rho_\ell\colon I_K \to \GL(T_\ell E)$ is expressed in terms of the Kummer pairing of $D$ in the same way as in our Lemma \ref{unimon}. One deduces that $\rho_\ell$ is unipotent of level at most 2, and that $\rho_\ell(P_K) = \{1\}$.
\end{rmk}

%medbreak
Let us return to the setting of Drinfeld modules. % $D$. % and $E$.
Following \S\ref{mainest} we consider Poonen's canonical local height 
$\infn{\cdot}$
on the $A$-module $\Eg(K)$.
We have seen that with a suitable normalization of the absolute value %$|\cdot|_\infty$ %\colon \Finf \to \R_{\geqslant 0}$)
this height is an $\infty$-adic seminorm.
Remarkably this restricts to a proper norm on every period lattice~$\Lat$:

\begin{lem}\label{latnorm}
The induced height $\infn{\cdot}\colon \Lat \to \Z_{\geqslant 0}$ has the following properties:
\begin{statements}
\item
$\infn{\lv} = 0$ if and only if $\lv = 0$,

\item\label{ultraineq}
$\infn{\lv + \lv'} \leqslant \max \{\infn{\lv}, \:\infn{\lv'} \}$,

% \item\label{ultraeq}
% Moreover, $\infn{\lv + \lv'} = \max \{ \infn{\lv}, \:\infn{\lv'} \}$
% when $\infn{\lv} \ne \infn{\lv'}$.

\item\label{latscale}
$\infn{a\lv} = |a|_\infty^s \, \infn{\lv}$ for all $a \in A$,
%\marked{Check that $s$ is understood.}

\item\label{latfin}
the subset $\Latb{b}$
is finite for every $b > 0$.
\end{statements}
The exponent $s$ in \ref{latscale} is the rank of the Drinfeld module~$\Eg$.
\end{lem}
\begin{proof}%[Proof of Lemma~\ref{latnorm}]
Properties \ref{ultraineq} and~\ref{latscale} hold for the ambient seminorm
and Property~\ref{latfin} follows since the lattice $\Lat$ is discrete in $\Eg(K)$.
We only need to check the first property. Note that $\infn{\lv} = 0$
if and only if $\lv\in\Eg(\OK)$.
The $A$-module $\Lat' = \Lat \cap \Eg(\OK)$ is finite since $\Lat$ is discrete in~$\Eg(K)$.
Hence $\Lat'$ is torsion, and thus zero as the enclosing module $\Lat$ is torsion-free.
\end{proof}

\medskip
This fits to our definition of a normed lattice (Definition~\ref{lattice})
with a caveat: The homogeneity property~\ref{latscale} holds for a differently
normalized $\infty$-adic absolute value.
When necessary we will correct for this discrepancy by considering the norm $\infn{\cdot}^{\frac{1}{s}}$. 

% The informal analogy between the local field $\Finf$ and the field of real numbers~$\R$
% hints that normed lattices of Lemma~\ref{latnorm} can be seen as function field counterparts of
% Euclidean lattices.
% Taking advantage of this observation we will use the last successive minimum
% of the period lattice to bound ramification of Tate modules, see 
% Theorem~\ref{lsm}. \marked{Add reference to Taguchi.}

\begin{thm}\label{tatevanish}
Let $E$ be a Drinfeld module of stable reduction over $\Spec K$
with period lattice $\Lat$ defined over~$K$.
Then there is an integer $\coi\geqslant 0$ such that
the monodromy representation $\rhad\colon J_K \to \GL(\Tad E)$ 
maps the ramification subgroup $\smash{J^\ram{\coi}_K}$ to~$1$.
Explicitly, if the period lattice
$\Lat$ is generated by elements $\lv_1, \dotsc, \lv_\lra$
then one can take
\begin{equation*}
\coi = \max \infn{\lv_i} + 1.
\end{equation*}
\end{thm}
\begin{proof}
By Theorem~\ref{kumimage}~\ref{rambound} each \KH homomorphism $\kumadu{\lv_i}$
vanishes on the subgroup $\smash{J^{\ram{\coi}}_K}$.
Since the Kummer pairing is $A$-linear in the first argument
it follows that the homomorphism $\kumadl$ vanishes on $\smash{J^{\ram{\coi}}_K}$ for
all $\lv \in \Lat$. Lemma~\ref{unimon} implies the result.
\end{proof}

% Old:
% \begin{dfn}\label{cond}
% The \emph{conductor} $\cond{E}$ is the minimal integer $\coi \geqslant 0$ such that
% $\rhad(J^{\ram{\coi}}_K) = 1$.
% \end{dfn}
% Alternative:
\medskip
In view of Theorem \ref{tatevanish} the following definition makes sense:

\begin{dfn}\label{cond}
The \emph{conductor} $\cond{E}$ is the least integer $m\ge0$ such that
$\rho(J_K^{\coi+1}) = 1$.
\end{dfn}

\begin{lem}\label{condzero}
We have $\cond{E} = 0$ if and only if the Drinfeld module $E$ has good reduction.
\end{lem}
\begin{proof}
We have $J_K^0 = J_K^1$ by Lemma \ref{ramcong} so $\cond{E} = 0$ if and only if
the inertia group $I_K$ acts trivially on $T_\fp E$ for all $\fp\ne\fpres$.
The claim thus follows from Takahashi's good reduction criterion \cite{takahashi}.
\end{proof}
% As $J_K^0 = J_K^1$ our conductor can never be equal to~$1$.
% This reflects the fact that
% in our setting the ramification, if it is present, is never tame.
%The name of this invariant was chosen for the following reason.
The invariant $\cond{E}$ relates to conductors of torsion point extensions in the following way.
\newcommand{\Kn}{K_\fn}%
\newcommand{\Knur}{{K_\fn}^{\kern-.9ex\circ}}%
For each proper ideal $\fn$ that is prime to~$\fpres$ %of~$E$
consider the extension $\Kn\subset K^s$ generated by the $\fn$-torsion points of~$E$,
which is to say, the group $\Gal(K^s/\Kn)$ is the kernel 
of the torsion points representation $G_K \to \GL(E[\fn])$.
Let $\Knur\subset\Kn$ be the maximal unramified subextension.

\begin{lem}\label{abcond}
The extensions $\Kn/\Knur$ are abelian and their conductors are related to $\cond{E}$ as follows.
If $\cond{E} \ne 0$ then we have an equality
\begin{equation*}
%\max_\fn\:\econd{\Kn/\Knur}  = \cond{E}
\cond{E} + 1 = \max_\fn\,\econd{\Kn/\Knur}.
\end{equation*}
If $\cond{E} = 0$ then $\econd{\Kn/\Knur} = 0$ for all $\fn$.
\end{lem}
\begin{proof}
%A choice of a generator of the module $\fn^{-1}/A$ identifies
The Galois representation $E[\fn]$ and $\ntor{E}$ are non-canonically isomorphic, cf.~\S\ref{predm}.
So the Galois group of the extension $\Kn/\Knur$ is the image
of the absolute inertia subgroup $I_K$ under the torsion points representation
$G_K \to \GL(\ntor{E})$. The latter is obtained from the adelic Tate module representation %$G_K \to \GL(\Tad E)$
by reduction modulo the ideal $\fn\Aad$.
Hence the extension $\Kn/\Knur$ is abelian by Corollary~\ref{abmon}.
When this extension is nontrivial its conductor and its highest break in upper numbering $b(\Kn/\Knur)$
are related in the following way:
\begin{equation*}
\econd{\Kn/\Knur} = b(\Kn/\Knur) + 1,
\end{equation*}
see \cite[\S4.2, Proposition~1]{serre-lcft}.
In view of this formula our claim is immediate.
\end{proof}

% % %\marked{Explain the shift by~$1$.}
The normalization of the conductor $\cond{E}$ thus differs by $1$ from the one of local class field theory.
In our case this choice leads to better-looking formulas.

\begin{lem}\label{condval}
%$\cond{E} \not\equiv 1 \pmod{p}$.\qed
The conductor $\cond{E}$ is either zero or a positive integer that is prime to~$p$.
\end{lem}
\begin{proof}
If $\coi \geqslant 0$ is an integer that is divisible by~$p$ %$\coi\equiv 0\pmod{p}$
then we have $J_K^{\coi} = J_K^{\coi+1}$ by Lemma~\ref{ramcong}.
\end{proof}

\begin{lem}\label{isogdepth}
The conductor is an isogeny invariant.
\end{lem}
\begin{proof}
We have to be careful as we do not fix the characteristic homomorphism of our Drinfeld modules.
Also, in this proof it will be more convenient to represent Drinfeld modules by homomorphisms
$A \to K[\tau]$.

Let $\varphi, \psi\colon A \to K[\tau]$ be Drinfeld modules
and let $f\colon \varphi \to \psi$ be an isogeny, i.e. a nonzero element of $K[\tau]$
satisfying $f \, \varphi(a) = \psi(a) \, f$ for all $a \in A$.
Write $f = g\,\tau^n$ with $g \in K[\tau]$ having nonzero constant coefficient.
We have a well-defined homomorphism $$\varphi'\colon A \to K[\tau], \quad a \mapsto \tau^n \varphi(a) \tau^{-n}.$$
One checks easily that $\varphi'$ is a Drinfeld module.
The identity $f \, \varphi(a) = \psi(a) \, f$ implies that
$g \, \varphi'(a) = \psi(a) \, g$ so
$g$ is an isogeny from $\varphi'$ to $\psi$.
Likewise $\tau^n$ is an isogeny from $\varphi$ to $\varphi'$.

As the constant coefficient of $g$ is nonzero it follows that the characteristic homomorphisms of $\varphi'$ and $\psi$ are the same. The characteristic homomorphisms $\chi_\varphi$ and $\chi_{\varphi'}$ are related by the formula
$$\chi_{\varphi'} = \sigma^n \circ \chi_\varphi$$
where $\sigma\colon K \to K$ is the $q$-th power map.
In particular if $\varphi$ has finite residual characteristic then so do $\varphi'$ and $\psi$,
and the residual characteristics of $\varphi$, $\varphi'$ and $\psi$ coincide.
We will denote these residual characteristics by $\fpres$ as usual.

Next, suppose that $\varphi$ has stable reduction.
It then follows by Lemma 6.1.5 (1) of \cite{papikian} that $\varphi'$ also has stable reduction.
Suppose further that $\psi$ also has stable reduction.
By Theorem 6.2.12 of \cite{papikian} the isogeny $g$ induces a Galois-equivariant injective morphism of period lattices
$g\colon \Lat_{\varphi'} \hookrightarrow \Lat_\psi$.
Assuming that $\Lat_\psi \subset K$ we thus conclude that $\Lat_{\varphi'} \subset K$.

It therefore suffices to prove our lemma in two cases: (1) $f = \tau^n$ and (2) the constant coefficient of $f$ is nonzero.
In the case $f = \tau^n$ the map $\sigma^n$ defines a Galois-equivariant isomorphism of torsion modules
$\varphi[a] \isoarrow \psi[a]$ for all $a \in A$. Invoking Lemma \ref{initial} we obtain a Galois-equivariant isomorphism
$\Tad(\varphi) \isoarrow \Tad(\psi)$, and the claim follows.

Finally we consider the case when the constant coefficient of $f$ is not zero.
We observed above that in this case the characteristic homomorphisms of $\varphi$ and $\psi$ coincide.
As a consequence there is an isogeny $f'\colon \psi \to \varphi$ which is dual to $f$ in the sense that
$f' f = \varphi(a)$ for some nonzero $a \in A$.
Since the multiplication by $a$ is injective on $\Tad(\varphi)$
the isogeny $f$ induces a Galois-equivariant injective morphism
$\Tad(\varphi) \hookrightarrow \Tad(\psi)$.
By definition of the conductor this implies that
$\cond{\varphi} \leqslant \cond{\psi}$.

The identity $f f' f = f \, \varphi(a) = \psi(a) \, f$ implies that
$f f' = \psi(a)$ because the ring $K[\tau]$ is a domain.
Repeating the argument above we deduce an inequality $\cond{\psi} \leqslant \cond{\varphi}$, and the lemma follows.
\end{proof}

\medskip
The ramification filtration of the group $J_K$ changes in a nontrivial way when
one passes to a finite separable extension $L/K$. It is not generally true that
$J_L^\coi$ maps into $J_K^\coi$, cf. the proof of Theorem~\ref{vanish} for a precise
statement.  The conductor thus depends on the choice of the field~$K$ which we
reflect in the notation $\cond{E}$.  Note that we only consider conductors of
Drinfeld modules of stable reduction under the extra assumption that the period
lattice is defined over~$K$.

Calculating the conductor is a hard problem in general.
However we have a full answer for a class of period lattices $\Lat$:

\begin{thm}\label{lsm}%
Let $E$ be a Drinfeld module of stable reduction over $\Spec K$
with period lattice $\Lat$ defined over~$K$.
Suppose that elements $\lv_1, \dotsc, \lv_\lra$ generate
an $A$-submodule of finite index in~$\Lat$. We then have an inequality
\begin{equation*}
\cond{E} \leqslant \max \infn{\lv_i}.
\end{equation*}
If in addition the maximum $b = \max \infn{\lv_i}$ is prime to~$p$
then the following holds:
\begin{statements}
\item\label{lsmall}
The adelic monodromy representation $\rhad\colon J_K \to \GL(\Tad E)$ maps the ramification subgroup $J_K^\ram{b}$
to an infinite subgroup. In particular
\end{statements}
\vspace*{-.25em}%
\begin{equation*}
\cond{E} = \max \infn{\lv_i}.
\end{equation*}
\begin{statements}[start=2]
\item\label{lsmone}
For every prime $\fp$ different from the residual characteristic~$\fpres$
the $\fp$-adic monodromy representation
$\rhp\colon J_K \to \GL(T_\fp E)$ maps the subgroup $J_K^\ram{b}$ to an infinite subgroup.
\end{statements}
\end{thm}
\begin{proof}%[Proof of Theorem~\ref{lsm}]
Since the conductor is invariant under isogenies we are free to assume that the period lattice $\Lat$ is generated by the vectors $\lv_1, \dotsc, \lv_\lra$.
The inequality $\cond{E} \leqslant \max \infn{\lv_i}$
then follows from Theorem~\ref{tatevanish}.

Next assume that the maximum $b = \max \infn{\lv_i}$ is prime to~$p$
and let $\lv_i$ be a vector on which the maximum is achieved.
Theorem~\ref{kumimage}~\ref{primetop} shows that the \KH 
homomorphism $\kumadu{\lv_i}$ maps the ramification subgroup
$J_K^b$ surjectively onto $\Tad \Eg$. Claim~\ref{lsmall} is then 
a consequence of Lemma~\ref{unimon}
and Claim~\ref{lsmone} follows since the natural projection
$\Tad\Eg \twoheadrightarrow T_\fp\Eg$ is surjective.
\end{proof}

One important consequence is that all the values of the conductor
not forbidden by Lemma~\ref{condval} do actually occur:

\begin{cor}
For every integer $\coi \geqslant 1$, $p\nmid\coi$, there is a Drinfeld module
over $\Spec K$ of rank~$2$ and conductor~$\coi$. In particular there are Drinfeld modules
of arbitrarily large conductor.\qed
\end{cor}

%\begin{ex}
%Let $n\ge 1$ be an integer and let $\pi\in K$ be a uniformizer.
%Consider a Drinfeld $\Fq[t]$-module $E_n$ over $\Spec K$ that is given by the homomorphism
%$$t \mapsto \pi + \tau + \pi \tau.$$
%We would like to calculate the conductor of $E_n$.
%
%Suppose $D$ is given by a homomorphism $t \mapsto \pi + (1+\epsilon)\tau$.
%Have:
%$$ e\cdot (\pi + (1+\epsilon)\tau) = (\pi + \tau + \pi\tau^2) \, e.$$
%Writing $e = 1 + \sum_{i\ge1} \alpha_i \tau^i$ we then have
%$$ \sum_i \alpha_i \pi^{q^i} \tau^i + \alpha_i (1+\epsilon^{q^i})\tau^{i+1} =
%\sum_i \alpha_i \pi \tau^i + \alpha_i^q \tau^{i+1} + \alpha_i^{q^2} \pi \tau^{i+2}.$$
%This gives the following equations:
%\begin{gather*}
%\alpha_1 \pi^q + \epsilon = \alpha_1 \pi, \\
%\alpha_{i+2} \pi^{q^{i+2}} + \alpha_{i+1} (1 + \epsilon^{q^{i+1}}) = \alpha_{i+2} \pi + \alpha_{i+1}^q + \alpha_i^{q^2} \pi, \quad i\ge0.
%\end{gather*}
%In particular, setting $u_i = (1 - \pi^{q^{i-1}})$ we obtain
%$$\epsilon = u_1 \alpha_1 \pi$$
%and
%$$u_{i+2} \alpha_{i+2} = \frac{\alpha_{i+1}(1 + \epsilon^{q^{i+1}}) - \alpha_{i+1}^q}{\pi} - \alpha_i^{q^2}.$$
%\end{ex}

By analogy with the classical theory of lattices one defines the $i$-th
successive minimum $\mu_i(\Lat)$ as the minimal real number $b \geqslant 0$
such that the set of vectors $\Latb{b}$ generates an $A$-submodule of rank at least~$i$.
Theorem~\ref{lsm} implies that
\begin{equation*}
\cond{E} \leqslant \mu_\lra(\Lat)
\end{equation*}
where $\mu_\lra(\Lat)$ is the last successive
minimum, $\lra = \rank \Lat$.

Finally let us bring the lattice theory of~\S\ref{lattices} to bear on
the problem of estimating the conductor:
\begin{thm}\label{condvol}
Let $E$ be a Drinfeld module of stable reduction over $\Spec K$
with period lattice $\Lat$ defined over~$K$.
Let $r$ be the rank of~$E$ and set $s: = r - \rank\Lat$.
Then we have an inequality:
\begin{equation*}
\cond{E} \leqslant \volu^s \, \econs^{s(r-s)}
\end{equation*}
where
\begin{itemize}
\item
$\volu$ is the volume of the normed lattice
$(\Lat,\,\infn{\cdot}^{\frac{1}{s}})$,

\item
$C = {\cco}^{\econx}$,

\item
$\fco$ is the field of constants of the coefficient field~$\F$,

\item
$g$ is the genus of the coefficient field~$\F$,

\item
$\rdeg$ is the degree of the place~$\infty$ over~$\fco$.
\end{itemize}
\end{thm}
\begin{proof}
Every nonzero period $\lv\in\Lat$ satisfies the inequality
$\infn{\lv}^{\frac{1}{s}} \geqslant 1$. Hence Theorem~\ref{volbound}
implies that the $A$-module $\Lat$ is generated by the subset
\begin{equation*}
\{\lv \in \Lat, \:\infn{\lv}^{\frac{1}{s}} \leqslant \volu \, \econs^{r-s} \}.
\end{equation*}
In view of this fact the bound on the conductor follows
from Theorem~\ref{tatevanish}.
\end{proof}

Whenever $\volu \cdot \econs^{r-s} \ne 1$ 
one can %erase the summand $+1$ from the estimate of Theorem~\ref{condvol} by appealing to Lemma~\ref{condval}.
tighten the bound in Theorem~\ref{condvol} by $1$ by appealing to Lemma~\ref{condval}.

\subsection{General results}

\begin{thm}\label{vanish}
Let $E$ be a Drinfeld module of finite residual characteristic~$\fpres$ over~$\Spec K$ 
and let $\rhad\colon G_K \to \GL(\Tad E)$ be its restricted adelic Tate module representation.
Then there is a rational number $u\geqslant 0$ such that $\rhad(I_K^u) = 1$.
\end{thm}
\begin{proof}
Let $L\subset K^s$ be a finite Galois extension of~$K$ such that
$E$ has stable reduction over $\Spec L$ and the local period lattice of~$E$
is defined over~$L$. Theorem~\ref{tatevanish} shows that
$\rhad(I_L^u) = 1$ for all rational numbers $u \gg 0$.
We have a short exact sequence of groups
\begin{equation*}
\begin{tikzcd}
1 \rar & I_L^u \rar & I_K^{\varphi(u)} \rar & \Gal(L/K)^{\varphi(u)} \rar & 1
\end{tikzcd}
\end{equation*}
where $\varphi$ is the lower-to-upper reindexing function of the extension $L/K$ \cite[Chapitre~IV, \S3]{corps}.
Taking a sufficiently large rational number~$u$ we ensure that
$\Gal(L/K)^{\varphi(u)} = 1$, and the claim follows.
\end{proof}

\begin{thm}
Let $E$ be a Drinfeld module of finite residual characteristic~$\fpres$ and rank~$r$ over $\Spec K$.
Suppose that the local period lattice of $E$ has rank $1$.
Then the group scheme $\GL(\Tad E)$ %over $\Spec(\Aad)$
contains a closed subgroup scheme $U \cong \smash{(\Ga)^{\times(r-1)}}$
such that the image of inertia is commensurable with $U(\Aad)$.
% The intersection $U \cap \rhad(I_K)$ has finite index in $U$ and
% in $\rhad(I_K)$.
\end{thm}
\begin{proof}
Since we are free to replace the field~$K$ by a finite separable extension
we assume that $E$ has stable reduction and that its local period lattice $\Lat$
is defined over~$K$. As in \S\ref{conductor} we have a short exact sequence
of Galois representations in finitely generated free $\Aad$-modules:
\begin{equation*}
\begin{tikzcd}
0 \rar & \Tad\Eg \rar & \Tad E \rar & \Aad \otimes_A \Lat \rar & 0.
\end{tikzcd}
\end{equation*}
This defines an $\Aad$-module flag on the middle term. Let $U$ be the corresponding
unipotent subgroup scheme of $\GL(\Tad E)$. By our assumption the period lattice $\Lat$ is
an invertible $A$-module. Hence the group scheme $U$ is isomorphic to $(\Ga)^{\times d}$
where
\begin{equation*}
d = \rank(\Tad\Eg) = r - 1.
\end{equation*}
%Let $\rhad\colon G_K \to \GL(\Tad E)$ be the restricted adelic Tate module representation.
The image of inertia $\monim$ is contained in the subgroup~$U(\Aad)$
since the representations $\Tad\Eg$ and $\Lat$ are unramified. 

Pick a nonzero period $\lv \in \Lat$.
As the period lattice $\Lat$ has rank~$1$ the evaluation at $\lv$ defines an open embedding:
\begin{equation*}
\Hom_A(\Lat,\:\Tad\Eg) \hookrightarrow \Tad\Eg.
\end{equation*}
By Theorem~\ref{kummeropen}
the homomorphism $\kumadl$ %\colon I_K \to \Tad\Eg$
has open image.
Lemma~\ref{unimon} implies that the subgroup
$\monim$ is open in~$U(\Aad) \isoarrow \Hom_A(\Lat,\,\Tad\Eg)$. Consequently the index of $\monim$
in~$U(\Aad)$ is finite.
\end{proof}

\end{document}